\documentclass[review]{elsarticle}
\usepackage{graphicx}
\usepackage{epstopdf}
\usepackage{lineno,hyperref}
\usepackage{geometry}
\usepackage{epsfig}
\usepackage{lineno,hyperref}
\usepackage{bookmark}
\usepackage{multirow}
\usepackage{amsfonts}
\usepackage{multirow}
\usepackage{algorithm}
\usepackage[figuresright]{rotating}
\usepackage{amscd}
\usepackage{mathrsfs}
\usepackage{booktabs,longtable}
\usepackage{indentfirst}
\usepackage{subfigure}
\usepackage{amstext}
\usepackage{amssymb}
\usepackage{fancyhdr}
\usepackage{amsmath}
\usepackage{appendix}
\usepackage{color}
\usepackage{color}
\usepackage{cleveref}
\usepackage{algorithmicx}
\usepackage{algpseudocode}
\usepackage{bm}
\usepackage{amsmath}
\usepackage{amsthm}
\modulolinenumbers[5]
\biboptions{sort&compress}
\newtheorem{theorem}{Theorem}
\newtheorem{lemma}{Lemma}
\newdefinition{remark}{Remark}
\newtheorem{definition}{Definition}

\newtheorem{example}{Example}
\newtheorem{corollary}{Corollary}

\journal{Journal of \LaTeX\ Templates}

\begin{document}

\begin{frontmatter}

\title{Subsampling for tensor least squares: Optimization and statistical perspectives
\tnoteref{mytitlenote}}
\tnotetext[mytitlenote]{The work is supported by the National Natural Science Foundation of China (No. 11671060) and the Natural Science Foundation of Chongqing, China (No. cstc2019jcyj-msxmX0267)\\ *Corresponding author\\ Email address: lihy.hy@gmail.com or hyli@cqu.edu.cn (Hanyu Li)}

\author{Ling Tang}
\author{Hanyu Li*}

\address{College of Mathematics and Statistics, Chongqing University, Chongqing 401331, P.R. China}
\begin{abstract}
In this paper, we investigate 
the random subsampling method for 
tensor least squares problem with respect to the popular t-product. 
From the optimization perspective, we present the error bounds in the sense of probability for the residual and solution obtained by the proposed method. 
From the statistical perspective, we derive 
the expressions of the conditional and unconditional expectations and variances for the solution, where the unconditional ones combine the model noises. Moreover, based on the unconditional variance, an optimal subsampling probability distribution is also found. 
Finally, the feasibility and effectiveness of the proposed method 
and the correctness of the theoretical results are verified by numerical experiments.
\end{abstract}


\end{frontmatter}


\section{Introduction}
The problem we consider in this paper is the following tensor least squares (TLS) problem:
\begin{align}\label{TLS}	\min_{\overrightarrow{\mathcal{B}}\in\mathbb{R}^{p\times 1 \times l}}f(\overrightarrow{\mathcal{B}})=\min_{\overrightarrow{\mathcal{B}}\in\mathbb{R}^{p\times 1 \times l}}\|\overrightarrow{\mathcal{Y}}-\mathcal{X}*\overrightarrow{\mathcal{B}}\|_{F}^{2},
\end{align}
where $\mathcal{X}\in\mathbb{R}^{n\times p\times l}$, $\overrightarrow{\mathcal{Y}}\in\mathbb{R}^{n\times 1\times l}$ and the operator $*$ is the t-product of tensors introduced by Kilmer and Martin in \cite{kilmer2011factorization}. Numerous applications, such as tensor low-rank approximation and decomposition \cite{tarzanagh2018fast}, digital image and signal processing \cite{kilmer2013third, soltani2016tensor}, and statistical models \cite{jin2017generalized}, etc., involve the problem (\ref{TLS}). It has attracted some attention, as seen in the recent works \cite{jin2017generalized,miao2020generalized, reichel2022tensor}. In the present paper, we mainly focus on the case of $n\gg p$, $n\gg l$, and $\text{rank}_{\text{t}}(\mathcal{X})=p$, where $\text{rank}_{\text{t}}(\mathcal{X})$ denotes the tubal rank of $\mathcal{X}$. In this case, the 
solution 
can be expressed as follows:
\begin{align}\label{solution}
\widetilde{\overrightarrow{\mathcal{B}}}_{\rm ols}=\mathcal{X}^{\dag}*\overrightarrow{\mathcal{Y}}=(\mathcal{X}^{T}*\mathcal{X})^{-1}*\mathcal{X}^{T}*\overrightarrow{\mathcal{Y}}.
\end{align}

It is expensive to compute the solution \eqref{solution}, especially for large-scale problem. As done for 
matrix least squares (MLS) problem, i.e., the case for $l=1$ in the problem (\ref{TLS}), we will investigate the random subsampling method for  TLS problem. At present, there are many works on the random subsampling method for MLS problem. For example, in 2006, Drineas et al. \cite{drineas2006sampling} presented and analyzed the leverage-based subsampling algorithms. 
Later, due to the high 
cost of leverage scores, many approaches for
amelioration were developed, such as computing them approximately \cite{drineas2012fast} or uniformizing them using random projections \cite{drineas2011faster, avron2010blendenpik, meng2014lsrn}. To our best knowledge, there is no related work for TLS problem reported in the literature.

The TLS problem \eqref{TLS} has a close relationship with the following tensor linear regression model:
\begin{align}\label{linearmodel}
\overrightarrow{\mathcal{Y}}=\mathcal{X}*\overrightarrow{\mathcal{B}}_{0}+\overrightarrow{\mathcal{E}},
\end{align}
where 
$\overrightarrow{\mathcal{B}}_{0}\in\mathbb{R}^{p\times 1\times l}$ is the true but unknown parameter, and
$\overrightarrow{\mathcal{E}}$ is the noise tensor with $\mathbf{E}[\overrightarrow{\mathcal{E}}]=\overrightarrow{\mathbf{0}}$ and $\mathbf{Cov}[\overrightarrow{\mathcal{E}}]=\sigma^{2}\mathcal{I}_{n}$. Actually, the solution \eqref{solution} is just 
the least squares estimator 
of 
$\overrightarrow{\mathcal{B}}_{0}$. Hence, besides discussing the random subsampling method for the problem \eqref{TLS} from the optimization perspective, i.e., obtaining the error bounds for residual and solution in probabilistic meaning, we will also consider the statistical properties of the aforementioned method under the setting of \eqref{linearmodel}. The expressions for the conditional and unconditional expectations and variances of the estimator will be given. Further, an optimal subsampling probability distribution is also derived by minimizing the trace of the unconditional variance. These results generalize the corresponding ones of random subsampling methods for MLS problem \cite{ma2014statistical,raskutti2016statistical,chi2022projector,dobriban2019asymptotics,JMLR}.

The remainder of this paper is organized as follows. Section \ref{Na_Pre} introduces some necessary notation and preliminaries. In Section \ref{sec_RTLS}, we present the random subsampling method for the TLS problem \eqref{TLS}. The analysis of the quality of the proposed method from optimization and  statistical perspectives are described in Sections \ref{sec_optimization} and  \ref{sec_statistical}, respectively. Section \ref{sub_ex} is  devoted to numerical experiments to test our method and verify our theoretical results. Finally, we give the conclusion of the whole paper. The detailed proofs of main lemmas and theorems are provided in the appendix, along with an additional algorithm.

\section{Notation and preliminaries}\label{Na_Pre}

\subsection{Notation}
Throughout this paper,  we will use lowercase letters (e.g., $x$), boldface lowercase letters (e.g., $\mathbf{x}$), capital letters (e.g., $X$) and calligraphic letters (e.g., $\mathcal{X}$)  to denote scalars, vectors, matrices and tensors, respectively. The Matlab style notation is used to represent fibers and slices of third-order tensor as outlined in Table \ref{No_thirdtensor}. For a positive integer $m$, let $[m]:=\{1,2,\cdots,m\}$.

\begin{table}[]
	\centering
	\fontsize{9}{9}\selectfont
	\caption{Summary of notation for third-order tensors.}\label{No_thirdtensor}
	\setlength{\tabcolsep}{0.6mm}{
	\begin{tabular}{|c c|}
		\hline
  $\mathcal{X}\in\mathbb{R}^{n\times p\times l}$ &  \\ \hline
$\mathcal{X}_{(i,j,k)}$  & The $(i,j,k)$-th element of $\mathcal{X}$ \\
$\mathcal{X}_{(i,:,k)}$ & The $(i,k)$-th row fiber of $\mathcal{X}$\\
	$\mathcal{X}_{(:,j,k)}$  & The $(j,k)$-th column fiber of $\mathcal{X}$\\
	$\mathcal{X}_{(i,j,:)}$& The $(i,j)$-th tube fiber of $\mathcal{X}$\\
	$\mathcal{X}_{(i,:,:)}$ & The $i$-th horizontal slice of $\mathcal{X}$\\
	$\mathcal{X}_{(:,j,:)}$ & The $j$-th lateral slice of $\mathcal{X}$\\
$\mathcal{X}_{(:,:,k)}$, $\mathcal{X}_{(k)}$  & The $k$-th frontal slice of $\mathcal{X}$ \\ \hline
	\end{tabular}
}
\end{table}

\subsection{t-product and the related basics}
In this paper, third-order tensors are also referred to as tubal matrices. The details are described in Definition \ref{def_tubal}.
\begin{definition}[see \cite{kilmer2013third}]\label{def_tubal}\rm
	An element $\mathbf{x}\in \mathbb{R}^{1\times 1\times l}$ is called a tubal scalar of length $l$ and the set consisting of all tubal scalars of length $l$ is denoted by $\mathbb{K}_{l}$; an element $\overrightarrow{\mathcal{X}} \in \mathbb{R}^{n\times 1\times l}$ 
	is  called a vector of tubal scalars of length $l$ with size $n$ and the corresponding set is denoted by $\mathbb{K}^{n}_{l}$;
	an element $\mathcal{X} \in \mathbb{R}^{n\times p\times l}$ is  called a matrix of tubal scalars of length $l$ with size $n\times p$ and the corresponding set is denoted by $\mathbb{K}^{n\times p}_{l}$.
\end{definition}

Now, we give the definition of the t-product.
\begin{definition}[t-product \cite{kilmer2011factorization}] \rm
	For $\mathcal{X}\in \mathbb{K}^{n\times p}_{l}$ and $\mathcal{Y}\in \mathbb{K}^{p\times r}_{l}$, their t-product $\mathcal{X}*\mathcal{Y}\in \mathbb{K}^{n\times r}_{l}$ is defined as
	$$\mathcal{X}*\mathcal{Y}=\text{fold}(\text{bcirc}(\mathcal{X})\text{unfold}(\mathcal{Y})),$$
	where
	$$\text{bcirc}(\mathcal{X}):=
	\begin{bmatrix}
		\mathcal{X}_{(1)} & \mathcal{X}_{(l)} & \cdots & \mathcal{X}_{(2)}\\
		\mathcal{X}_{(2)} & \mathcal{X}_{(1)} & \cdots & \mathcal{X}_{(3)}  \\
		\vdots & \vdots & \ddots & \vdots \\
		\mathcal{X}_{(l)} & \mathcal{X}_{(l-1)} & \cdots & \mathcal{X}_{(1)} \\
	\end{bmatrix}
	,~~\text{unfold}(\mathcal{Y}):=\begin{bmatrix}
		\mathcal{Y}_{(1)} \\
		\mathcal{Y}_{(2)} \\
		\vdots \\
		\mathcal{Y}_{(l)}
	\end{bmatrix}
	, \textrm{ and } \text{fold}(\text{unfold}(\mathcal{Y})):=\mathcal{Y}.$$
\end{definition}

The t-product can be computed by the discrete Fourier transform (DFT), as shown in Algorithm \ref{t-product}.
\begin{algorithm}[htbp]
	\caption{t-product $\mathcal{Z}=\mathcal{X}*\mathcal{ Y}$ in the Fourier domain\cite{kilmer2011factorization}} \label{t-product}
	\hspace*{0.02in} {\bf Input:} 
	$\mathcal{X}\in \mathbb{K}^{n\times p }_{l}$, $\mathcal{Y}\in \mathbb{K}^{p\times r }_{l}$
	\begin{algorithmic}[1]
		\State $\widehat{\mathcal{ X}}=\texttt{fft}(\mathcal{ X},[~],3)$ and $\widehat{\mathcal{Y}}=\texttt{fft}(\mathcal{ Y},[~],3)$
		\For{$k=1,2,\cdots,l$}
		\State $\widehat{\mathcal{ Z}}_{(k)}=\widehat{\mathcal{ X}}_{(k)}\widehat{\mathcal{ Y}}_{(k)}$
		\EndFor
		\State $\mathcal{ Z}=\texttt{ifft}(\widehat{\mathcal{ Z}},[~],3)$.
	\end{algorithmic}
	\hspace*{0.02in} {\bf Output:} $\mathcal{Z}\in \mathbb{K}^{n\times r}_{l}$
\end{algorithm}
Moreover, due to the special structure of the DFT, the $l$ matrix-matrix multiplications in Algorithm \ref{t-product} can be reduced to $\lceil\frac{l+1}{2}\rceil$, where $\lceil n \rceil$ means the nearest integer number larger than or equal to $n$. That is, we can replace the "for" loop in Algorithm \ref{t-product} with the following computations \cite{lu2019tensor}:
\begin{equation*}
	\left\{
	\begin{array}{lcl}
		\widehat{\mathcal{ Z}}_{(k)}=\widehat{\mathcal{ X}}_{(k)}\widehat{\mathcal{ Y}}_{(k)},&&\text{for}~k=1,2,\cdots,\lceil\frac{l+1}{2}\rceil,\\
		\widehat{\mathcal{Z}}_{(k)}=\text{conj}(\widehat{\mathcal{Z}}_{(l-k+2)}), & &\text{for}~ k=\lceil\frac{l+1}{2}\rceil+1,\cdots,l.
	\end{array} \right.
\end{equation*}
We will use them in all our simulations.

Next, we will review some essential definitions, lemmas and properties related to t-product, which will be utilized for the following sections. For details, refer to \cite{kilmer2011factorization, jin2017generalized, qi2021t, zheng2021t, kilmer2013third}.
\begin{definition}[identity tubal matrix \cite{kilmer2011factorization}]\rm
	The identity tubal matrix $\mathcal{I}_{n}\in \mathbb{K}^{n\times n }_{l}$ is de?ned as a tubal matrix whose first frontal slice is the $n\times n$ identity matrix, and other frontal slices are all zeros.
\end{definition}

\begin{definition}[inverse \cite{kilmer2011factorization}]\rm
	For $\mathcal{X}\in \mathbb{K}^{n\times n}_{l}$, if there exists $\mathcal{Y} \in \mathbb{K}^{n\times n}_{l}$ satisfying
	$$\mathcal{X}*\mathcal{Y}=\mathcal{I}_{n}\quad \textrm{ and }\quad \mathcal{Y}*\mathcal{X}=\mathcal{I}_{n},$$
	then $\mathcal{X}$ is said to be invertible, and $\mathcal{Y}$ is called as the inverse of $\mathcal{X}$, which is denoted by $\mathcal{X}^{-1}$.
\end{definition}

\begin{definition}[Moore-Penrose inverse \cite{jin2017generalized}]\rm
	For $\mathcal{X}\in \mathbb{K}^{n\times p}_{l}$, if there exists $\mathcal{Y} \in \mathbb{K}^{p\times n}_{l}$ satisfying
	$$\mathcal{X}*\mathcal{Y}*\mathcal{X}=\mathcal{X},~~~\mathcal{Y}*\mathcal{X}*\mathcal{Y}=\mathcal{Y},~~~(\mathcal{X}*\mathcal{Y})^{T}=\mathcal{X}*\mathcal{Y},~~~(\mathcal{Y}*\mathcal{X})^{T}=\mathcal{Y}*\mathcal{X},$$
	then $\mathcal{Y}$ is called as the Moore-Penrose inverse of $\mathcal{X}$, which is denoted by $\mathcal{X}^{\dag}$.
\end{definition}


\begin{lemma}[see \cite{jin2017generalized}]
	The Moore-Penrose inverse of any $\mathcal{X}\in \mathbb{K}^{n\times p}_{l}$ exists and is unique, and if $\mathcal{X}$ is invertible, then $\mathcal{X}^{\dag}=\mathcal{X}^{-1}$.
\end{lemma}

\begin{definition}[transpose, slice transpose and reverse \cite{kilmer2011factorization,tang2022sketch}]\rm
	Suppose $\mathcal{X}\in \mathbb{K}^{n\times p}_{l}$. The transpose $\mathcal{X}^{T}$ is defined by transposing all of the frontal slices of $\mathcal{X}$ and then reversing the order of the transposed frontal slices from $2$ to $l$, that is
	\begin{equation*}
		\left\{
		\begin{array}{lcl}
			(\mathcal{ X}^{T})_{(1)}=\mathcal{ X}_{(1)}^{T}, && \\
			(\mathcal{ X}^{T})_{(k)}=\mathcal{ X}_{(l-k+2)}^{T}, & &\text{for}~ k=2,3,\cdots,l,
		\end{array} \right.
	\end{equation*}
where $\mathcal{ X}_{(k)}^{T}$  represents the transport of $\mathcal{ X}_{(k)}$ for $k=1,2,\cdots,l$. The slice transpose $\mathcal{X}^{ST}$ is defined by transposing all of the frontal slices of $\mathcal{X}$, that is
\begin{equation*}
	(\mathcal{ X}^{ST})_{(k)}=\mathcal{ X}_{(k)}^{T}, \quad\text{for}~ k=1,2,\cdots,l.
\end{equation*}
And, the reverse $\mathcal{X}^{R}$ is defined by reversing the order of the frontal slices $\mathcal{X}$ from $2$ to $l$, that is
\begin{equation*}
	\left\{
	\begin{array}{lcl}
		(\mathcal{ X}^{R})_{(1)}=\mathcal{ X}_{(1)}, && \\
		(\mathcal{ X}^{R})_{(k)}=\mathcal{ X}_{(l-k+2)}, & &\text{for}~ k=2,3,\cdots,l.
	\end{array} \right.
\end{equation*}
\end{definition}

\begin{definition}[T-symmetric \cite{kilmer2011factorization}]\rm
	For $\mathcal{X}\in \mathbb{K}^{n\times n}_{l}$, it is T-symmetric if $\mathcal{X}=\mathcal{X}^{T}$.
\end{definition}
\begin{definition}[T-symmetric T-positive (semi)definite \cite{zheng2021t}]\rm
	For $\mathcal{X}\in \mathbb{K}^{n\times n }_{l}$, 
	it is T-symmetric T-positive (semi)definite if $\mathcal{X}$ is  T-symmetric and $\langle\overrightarrow{\mathcal{Y}},\mathcal{X}*\overrightarrow{\mathcal{Y}}\rangle>(\geq) 0$ holds for any nonzero $\overrightarrow{\mathcal{Y}}\in \mathbb{K}^{n}_{l}$ (for any $\overrightarrow{\mathcal{Y}}\in \mathbb{K}^{n}_{l}$).
\end{definition}

\begin{definition} [orthogonal tubal matrix \cite{kilmer2011factorization,zhang2018randomizedddd}]\rm
	For $\mathcal{X}\in \mathbb{K}^{n\times p}_{l}$, it is orthogonal if $n=p$ and $$\mathcal{X}^{T}*\mathcal{X}=\mathcal{I}_{n}.$$
 If the above equation holds but $n>p$, then $\mathcal{X}$ is said to be partially orthogonal.
\end{definition}
\begin{definition} [f-diagonal tubal matrix\cite{kilmer2011factorization}]\rm
For $\mathcal{X}\in \mathbb{K}^{n\times p}_{l}$, it is f-diagonal if each of its frontal slices is a diagonal matrix.
\end{definition}
\begin{definition} [t-SVD and tubal rank\cite{kilmer2011factorization,zhang2014novel}]\rm
For $\mathcal{X}\in \mathbb{K}^{n\times p}_{l}$, its t-SVD is $$\mathcal{X}=\mathcal{U}*\Sigma*\mathcal{V}^{T},$$
where $\mathcal{U}\in \mathbb{K}^{n\times n}_{l}$ and $\mathcal{V}\in \mathbb{K}^{p\times p}_{l}$ are orthogonal tubal matrices, and $\Sigma\in \mathbb{K}^{n\times p}_{l}$ is a f-diagonal tubal matrix.
And the tubal rank of $\mathcal{X}$ is denoted by $\text{rank}_{t}(\mathcal{X})$ and defined as the number of nonzero singular tubal scalars of $\Sigma$, that is, $$\text{rank}_{t}(\mathcal{X})=\text{card}\{i:\Sigma_{(i,i,:)}\neq \mathbf{0}\},$$
where $\text{card}$ denotes the cardinality of a set and  $\mathbf{0}$ is zero tubal scalar.
\end{definition}
\begin{definition} [thin t-SVD\cite{zhang2018randomizedddd}]\rm
For $\mathcal{X}\in \mathbb{K}^{n\times p}_{l}$ with tubal rank $r$, its thin t-SVD is $$\mathcal{X}=\mathcal{U}_{r}*\Sigma_{r}*\mathcal{V}_{r}^{T},$$
where $\mathcal{U}_{r}\in \mathbb{K}^{n\times r}_{l}$ and $\mathcal{V}_{r}\in \mathbb{K}^{p\times r}_{l}$ are partially orthogonal tubal matrices, and $\Sigma_{r}\in \mathbb{K}^{r\times r}_{l}$ is a f-diagonal tubal matrix.
\end{definition}
\begin{definition} [t-vectorization \cite{tang2022sketch}]\rm
	For $\mathcal{X}\in\mathbb{K}_{l}^{n\times p}$, its t-vectorization is defined as 
	$$\text{vec}_{t}(\mathcal{X}):=\begin{bmatrix}
		\mathcal{ X}_{(:,1,:)}  \\
		\vdots   \\
		\mathcal{ X}_{(:,p,:)} \\
	\end{bmatrix}\in\mathbb{K}_{l}^{np}.$$
\end{definition}
\begin{definition} [t-Kronecker product \cite{tang2022sketch}]\rm
	For $\mathcal{ X}\in\mathbb{K}_{l}^{n\times p}$ and $\mathcal{ Y}\in\mathbb{K}_{l}^{r\times s}$, their t-Kronecker product is defined as 
	$$\mathcal{X}\otimes_{t}\mathcal{Y}:=\begin{bmatrix}
		\mathcal{ X}_{(1,1,:)}*\mathcal{ Y} &  \cdots & \mathcal{ X}_{(1,p,:)} *\mathcal{ Y} \\
		\vdots 	&    \ddots &  	\vdots  \\
		\mathcal{ X}_{(n,1,:)}*\mathcal{ Y} &  \cdots  &  \mathcal{ X}_{(n,p,:)} *\mathcal{ Y}\\
	\end{bmatrix}\in\mathbb{K}_{l}^{nr\times ps}.$$
\end{definition}
\begin{lemma}[see \cite{tang2022sketch}]
	Suppose $\mathcal{ X}$, $\mathcal{ Y}$, $\mathcal{ Z}$ and $\mathcal{W}$ are tubal matrices of any multiplicable dimension. Then
	\begin{enumerate}
		\item $\text{\rm vec}_{t}(\mathcal{X}*\mathcal{ Y}*\mathcal{ Z})=(\mathcal{ Z}^{ST}\otimes_{t}\mathcal{ X})*\text{\rm vec}_{t}(\mathcal{ Y})$;
		\item $(\mathcal{ X}\otimes_{t}\mathcal{ Y})^{T}=\mathcal{ X}^{T}\otimes_{t}\mathcal{ Y}^{T}$,  $(\mathcal{ X}\otimes_{t}\mathcal{ Y})^{ST}=\mathcal{ X}^{ST}\otimes_{t}\mathcal{ Y}^{ST}$, $(\mathcal{ X}\otimes_{t}\mathcal{ Y})^{R}=\mathcal{ X}^{R}\otimes_{t}\mathcal{ Y}^{R}$;
		
		\item $(\mathcal{ X}\otimes_{t}\mathcal{ Y})^{\dag}=\mathcal{ X}^{\dag}\otimes_{t}\mathcal{ Y}^{\dag}$;
		
		\item $(\mathcal{ X}*\mathcal{ Y})\otimes_{t}(\mathcal{ Z}*\mathcal{ W})=(\mathcal{ X}\otimes_{t}\mathcal{ Z})*(\mathcal{ Y}\otimes_{t}\mathcal{ W})$;
		
		\item  $(\mathcal{ X}+\mathcal{ Y})\otimes_{t}\mathcal{ Z}=(\mathcal{ X}\otimes_{t}\mathcal{ Z})+(\mathcal{ Y}\otimes_{t}\mathcal{ Z})$.
	\end{enumerate}
\end{lemma}

	\begin{definition}[random tubal matrix \cite{jin2017generalized}]\rm
	A random tubal matrix is defined as a tubal matrix whose elements are all random variables.
\end{definition}
\begin{definition}[expectation and variance \cite{jin2017generalized}]\rm
	Supposing $\mathcal{X}\in\mathbb{K}^{n\times p}_{l}$ is a random tubal matrix, its expectation and variance are defined as
	$$\mathbf{E}[\mathcal{X}]=(\mathbf{E}[\mathcal{X}_{(i,j,k)}])\in\mathbb{K}^{n\times p}_{l},$$
	and
	$$\mathbf{Var}[\mathcal{X}]=\mathbf{E}[(\mathcal{X}-\mathbf{E}[\mathcal{X}])*(\mathcal{X}-\mathbf{E}[\mathcal{X}])^{T}]\in\mathbb{K}^{n\times n}_{l},$$
	respectively.
\end{definition}

\begin{definition}[covariance \cite{jin2017generalized}]\rm
	Supposing $\mathcal{X}\in\mathbb{K}^{n\times p}_{l}$ and $\mathcal{Y}\in\mathbb{K}^{n\times p}_{l}$ are random tubal matrices, their covariance is defined as $$\mathbf{Cov}(\mathcal{X},\mathcal{Y})=\mathbf{E}[(\mathcal{X}-\mathbf{E}[\mathcal{X}])*(\mathcal{Y}-\mathbf{E}[\mathcal{Y}])^{T}]\in\mathbb{K}^{n\times n}_{l}.$$
\end{definition}

\begin{lemma} Suppose $\mathcal{X}$ and $\mathcal{Y}$ are random tubal matrices, and $\mathcal{A}$ and $\mathcal{B}$ are constant tubal matrices. Then 
	\begin{enumerate}
		\item $\mathbf{E}[\mathcal{A}*\mathcal{X}*\mathcal{B}]=\mathcal{A}*\mathbf{E}[\mathcal{X}]*\mathcal{B}$, $\mathbf{Var}[\mathcal{A}*\mathcal{X}]=\mathcal{A}*\mathbf{Var}[\mathcal{X}]*\mathcal{A}^{T}$;
		\item $\mathbf{Var}[\mathcal{X}]=\mathbf{E}[\mathcal{X}*\mathcal{X}^{T}]-\mathbf{E}[\mathcal{X}]*\mathbf{E}[\mathcal{X}]^{T}$,  $\mathbf{Var}[\mathcal{X}\mid\mathcal{Y}]=\mathbf{E}[\mathcal{X}*\mathcal{X}^{T}\mid\mathcal{Y}]-\mathbf{E}[\mathcal{X}\mid\mathcal{Y}]*\mathbf{E}[\mathcal{X}\mid\mathcal{Y}]^{T}$;
		\item $\mathbf{E}[\mathcal{X}]=\mathbf{E}[\mathbf{E}[\mathcal{X}\mid\mathcal{Y}]]$, $\mathbf{Var}[\mathcal{X}]=\mathbf{Var}[\mathbf{E}[\mathcal{X}\mid\mathcal{Y}]]+\mathbf{E}[\mathbf{Var}[\mathcal{X}\mid\mathcal{Y}]]$.
	\end{enumerate}
\end{lemma}
\begin{proof}
	The proof is straightforward, but tedious, so we omit it.
\end{proof}
\begin{definition}[tubal scalar function\cite{miao2020generalized}]\label{def;itubalscalarfunction}\rm
Supposing $f: \mathbb{C}\to \mathbb{C}$ is a scalar function, its induced tubal scalar function $ \mathbf{f}: \mathbb{K}_{l}\to \mathbb{K}_{l}$
is defined as:
\begin{align*}
	\mathbf{f}(\mathbf{x})=\text{bcirc}^{-1}\left(F_{l}\begin{bmatrix}
		f(\widehat{\mathbf{x}}_{(1)}) &    & \\
			&    \ddots &  \\
	&  &   f(\widehat{\mathbf{x}}_{(l)})\\
	\end{bmatrix}F_{l}^{H} \right),
\end{align*}
where $F_{l}$ is the unitary DFT matrix and
\begin{align*}
\begin{bmatrix}
	\widehat{\mathbf{x}}_{(1)} &    & \\
	&    \ddots &  \\
	&    & \widehat{\mathbf{x}}_{(l)}\\
\end{bmatrix}=F_{l}^{H}\text{bcirc}(\mathbf{x})F_{l}.
\end{align*}
\end{definition}
\begin{definition}[derivative of tubal scalar function]\label{def;deritubalscalarfunction}\rm
Supposing $ \mathbf{f}: \mathbb{K}_{l}\to \mathbb{K}_{l}$ is a tubal scalar function, for $\mathbf{x}, \Delta\mathbf{x}\in\mathbb{K}_{l}$, if there exists $\mathbf{g}: \mathbb{K}_{l}\to \mathbb{K}_{l}$ satisfying
\begin{align*}
	\lim_{\Delta\mathbf{x}\rightarrow\mathbf{0}}\frac{\| \mathbf{f}(\mathbf{x}+\Delta\mathbf{x})-\mathbf{f}(\mathbf{x})-\mathbf{g}(\Delta\mathbf{x})\|_{F}}{\| \Delta\mathbf{x}\|_{F}}=0,
\end{align*}
where $\|\cdot\|_{F}$ denotes the Frobenius norm, then $\mathbf{f}$ is said to be differentiable at $\mathbf{x}$, and $\mathbf{g}(\mathbf{x})$ is called as the derivative of $\mathbf{f}$ at $\mathbf{x}$, which is denoted by $\mathbf{f}'(\mathbf{x})$.
\end{definition}

From Definitions \ref{def;itubalscalarfunction} and \ref{def;deritubalscalarfunction}, we can check that
\begin{align*}
	\mathbf{f}'(\mathbf{x})=\frac{\partial\mathbf{f}}{\partial\mathbf{x}}=\text{bcirc}^{-1}\left(F_{l}\begin{bmatrix}
		f'(\widehat{\mathbf{x}}_{(1)})	  &  & \\
		&    \ddots &  \\
		&    & f'(\widehat{\mathbf{x}}_{(l)})\\
	\end{bmatrix}F_{l}^{H}\right),
\end{align*}
where $f'(x)=\frac{\partial f(x)}{\partial x}$ is the derivative of $f$ at $x$.

\begin{definition}[gradient]\label{gradient}\rm
	Given a function $\mathbf{f}: \mathbb{K}_{l}^{m}\to \mathbb{K}_{l}$, for $\overrightarrow{\mathcal{X}}=(\mathbf{x}_{1},\cdots,\mathbf{x}_{m})^{ST}, \Delta\overrightarrow{\mathcal{X}}\in\mathbb{K}_{l}^{m}$, if there exists $\overrightarrow{\mathcal{G}}\in\mathbb{K}_{l}^{m}$ satisfying
	\begin{align*}
		\lim_{\Delta\overrightarrow{\mathcal{X}}\rightarrow\overrightarrow{\mathbf{0}}}\frac{\| \mathbf{f}(\overrightarrow{\mathcal{X}}+\Delta\overrightarrow{\mathcal{X}})-\mathbf{f}(\overrightarrow{\mathcal{X}})-\overrightarrow{\mathcal{G}}^{ST}*\Delta\overrightarrow{\mathcal{X}}\|_{F}}{\| \Delta\overrightarrow{\mathcal{X}}\|_{F}}=0,
	\end{align*}
where $\overrightarrow{\mathbf{0}}$ is zero tubal vector, then $\mathbf{f}$ is said to be differentiable at $\overrightarrow{\mathcal{X}}$, and $\overrightarrow{\mathcal{G}}$ is called as the gradient of $\mathbf{f}$ at $\overrightarrow{\mathcal{X}}$, which is denoted by $\triangledown\mathbf{f}(\overrightarrow{\mathcal{X}})$.
\end{definition}

From Definition \ref{gradient}, we can check that
\begin{align*}
\triangledown\mathbf{f}(\overrightarrow{\mathcal{X}})=\frac{\partial\mathbf{f}}{\partial\overrightarrow{\mathcal{X}}}=\left(\frac{\partial\mathbf{f}}{\partial\mathbf{x}_{1}}, \cdots, \frac{\partial\mathbf{f}}{\partial\mathbf{x}_{m}}\right)^{ST}\in\mathbb{K}_{l}^{m},
\end{align*}
where $\overrightarrow{\mathcal{X}}=(\mathbf{x}_{1}, \cdots, \mathbf{x}_{m})^{ST }\in\mathbb{K}^{m}_{l}$.
\begin{definition}[Jacobian tubal matrix]\label{jacobian}\rm
Given a function  $\overrightarrow{\mathcal{F}}(\overrightarrow{\mathcal{X}}): \mathbb{K}_{l}^{m}\to \mathbb{K}^{p}_{l}$, its Jacobian tubal matrix is defined as 
\begin{align*}
	D\overrightarrow{\mathcal{F}}=\frac{\partial\overrightarrow{\mathcal{F}}}{\partial\overrightarrow{\mathcal{X}}^{ST}}=\left(\frac{\partial\overrightarrow{\mathcal{F}}}{\partial\mathbf{x}_{1}}, \cdots, \frac{\partial\overrightarrow{\mathcal{F}}}{\partial\mathbf{x}_{m}}\right)=\begin{bmatrix}
	\frac{\partial\mathbf{f}_{1}}{\partial\mathbf{x}_{1}}	& \cdots   &\frac{\partial\mathbf{f}_{1}}{\partial\mathbf{x}_{m}} \\
	\vdots 	&    \ddots & \vdots  \\
	\frac{\partial\mathbf{f}_{p}}{\partial\mathbf{x}_{1}}	&\cdots  & \frac{\partial\mathbf{f}_{p}}{\partial\mathbf{x}_{m}}\\
	\end{bmatrix}\in\mathbb{K}_{l}^{p\times m}.
\end{align*}

\end{definition}

By Definition \ref{jacobian}, for a function $\mathcal{F}(\overrightarrow{\mathcal{X}}): \mathbb{K}^{m}_{l}\to \mathbb{K}^{p\times q}_{l}$,  we have that the Jacobian tubal matrix of $\text{vec}_{t}(\mathcal{F})$ is
\begin{align*}
	D\text{vec}_{t}(\mathcal{F})=\frac{\partial\text{vec}_{t}(\mathcal{F})}{\partial\overrightarrow{\mathcal{X}}^{ST}}=\left(\frac{\partial\text{vec}_{t}(\mathcal{F})}{\partial\mathbf{x}_{1}}, \cdots, \frac{\partial\text{vec}_{t}(\mathcal{F})}{\partial\mathbf{x}_{m}}\right).
\end{align*}
\begin{lemma}
	Given a function $\mathcal{F}(\overrightarrow{\mathcal{X}}): \mathbb{K}^{m}_{l}\to \mathbb{K}^{p\times q}_{l}$, suppose $\mathcal{A}\in\mathbb{K}^{r\times p}_{l}$ and $\mathcal{B}\in\mathbb{K}^{q\times s}_{l}$ are constant tubal matrices. Then
	\begin{align*}
		\frac{\partial\text{vec}_{t}(\mathcal{A}*\mathcal{F}*\mathcal{B})}{\partial\overrightarrow{\mathcal{X}}^{ST}}=(\mathcal{B}^{ST}\otimes_{t}\mathcal{A})*\frac{\partial\text{vec}_{t}(\mathcal{F})}{\partial\overrightarrow{\mathcal{X}}^{ST}},\\
		\frac{\partial\text{vec}_{t}(\mathcal{F}^{-1})}{\partial\overrightarrow{\mathcal{X}}^{ST}}=-((\mathcal{F}^{-1})^{ST}\otimes_{t}\mathcal{F}^{-1})*\frac{\partial\text{vec}_{t}(\mathcal{F})}{\partial\overrightarrow{\mathcal{X}}^{ST}}.
	\end{align*}
\end{lemma}
\begin{proof}
	The proof is straightforward, but tedious, so we omit it.
\end{proof}

Finally, we provide the definition of the random sampling tubal matrix.
\begin{definition}[random sampling tubal matrix \cite{tarzanagh2018fast}]\label{samplingrandom}\rm
	Assume that a random sampling is implemented for choosing $\tau$ lateral slices, one in each of independent and identical distributed (i.i.d.) trials.  A tubal matrix $\mathcal{S} \in\mathbb{K}^{n\times \tau}_{l}$ is called a random sampling tubal matrix, when 
	$\mathcal{S}_{(i,j,1)}=1
	$ if the $i$-th lateral slice is picked in the $j$-th independent trial and $\mathcal{S}_{(i,j,1)}=0$ otherwise, and other frontal slices are all zeros.
\end{definition}

\section{Random subsampling method for the TLS problem}\label{sec_RTLS}
The main idea of the random subsampling method is to first randomly choose $\tau\geq p$ subsamples of the data, i.e., the horizontal slices of $\mathcal{X}$ and the corresponding tubal scalars of $\overrightarrow{\mathcal{Y}}$, using a probability distribution $\{\pi_{i}\}_{i=1}^{n}$. And then form the following weighted TLS problem on the subsamples: 
	\begin{align}\label{WSTLS}
			\min_{\overrightarrow{\mathcal{B}}\in\mathbb{K}^{p}_{l}}\|\mathcal{D}*\mathcal{S}^{T}*\overrightarrow{\mathcal{Y}}-\mathcal{D}*\mathcal{S}^{T}*\mathcal{X}*\overrightarrow{\mathcal{B}}\|_{F}^{2},
	\end{align}
	where $\mathcal{S}\in\mathbb{K}_{l}^{n\times \tau}$ is a random sampling tubal matrix defined in Definition \ref{samplingrandom}, and $\mathcal{D}\in\mathbb{K}_{l}^{\tau\times \tau}$ is a rescaling tubal matrix satisfying that $\mathcal{D}_{(i,i,1)}=\frac{1}{\sqrt{\tau\pi_{k}}}$ if the $k$-th horizontal slice of $\mathcal{X}$ is chosen in the $i$-th random trial, and the others are all zeros.
	By solving (\ref{WSTLS}), an approximation solution of (\ref{TLS}) can be obtained as follows:
	$$\widetilde{\overrightarrow{\mathcal{B}}}_{ \mathcal{W}}=\arg\min_{\overrightarrow{\mathcal{B}}\in\mathbb{K}^{p}_{l}}\|\mathcal{D}*\mathcal{S}^{T}*\overrightarrow{\mathcal{Y}}-\mathcal{D}*\mathcal{S}^{T}*\mathcal{X}*\overrightarrow{\mathcal{B}}\|_{F}^{2}=(\mathcal{D}*\mathcal{S}^{T}*\mathcal{X})^{\dag}*\mathcal{D}*\mathcal{S}^{T}*\overrightarrow{\mathcal{Y}}.$$
	If $\mathcal{D}*\mathcal{S}^{T}*\mathcal{X}$ is further assumed to be of full tubal rank, then $\widetilde{\overrightarrow{\mathcal{B}}}_{ \mathcal{W}}$ can be expressed as
	\begin{align}\label{sol_WTLS}
		\widetilde{\overrightarrow{\mathcal{B}}}_{\mathcal{W}}=(\mathcal{X}^{T}*\mathcal{S}*\mathcal{D}^{2}*\mathcal{S}^{T}*\mathcal{X})^{-1}*\mathcal{X}^{T}*\mathcal{S}*\mathcal{D}^{2}*\mathcal{S}^{T}*\overrightarrow{\mathcal{Y}}=(\mathcal{X}^{T}*\mathcal{W}*\mathcal{X})^{-1}*\mathcal{X}^{T}*\mathcal{W}*\overrightarrow{\mathcal{Y}},
	\end{align}	
	where $\mathcal{W} = \mathcal{S}*\mathcal{D}^{2}*\mathcal{S}^{T}\in\mathbb{K}^{\tau \times \tau}_{l}$. 
	We summarize the random subsampling method in Algorithm \ref{Subsample_TLS}, and call it Subsampling TLS.
	\begin{algorithm}[htbp]
		\caption{Subsampling TLS} \label{Subsample_TLS}
		\hspace*{0.02in} {\bf Input:}
		$\mathcal{X}\in \mathbb{K}^{n\times p }_{l}$, $\overrightarrow{\mathcal{Y}}\in \mathbb{K}^{n }_{l}$, subsampling size $\tau\geq p$  and a probability distribution $\{\pi_{i}\}_{i=1}^{n}$
		\begin{algorithmic}[1]
			\State Initialize $\mathcal{S}\in\mathbb{K}_{l}^{n\times \tau}$ and $\mathcal{D}\in\mathbb{K}_{l}^{\tau \times \tau}$ to zero tubal matrices
			\For{$t=1,\cdots,\tau$}
			\State Pick $i_{t}\in[n]$, where $\mathbf{P}(i_{t}=i)=\pi_{i}$
		    \State $\mathcal{D}_{(t,t,1)}=\frac{1}{\sqrt{\tau\pi_{i_{t}}}}$
		    \State $\mathcal{S}_{(i_{t},t,1)}=1$
			\EndFor
			\State Form the weighted TLS subproblem (\ref{WSTLS})
			\State Solve the weighted TLS subproblem (\ref{WSTLS})
		\end{algorithmic}
		\hspace*{0.02in} {\bf Output:} the approximation solution $\widetilde{\overrightarrow{\mathcal{B}}}_{\mathcal{W}}$
	\end{algorithm}
  \begin{remark}
	Compared with the original TLS problem (\ref{TLS}), the weighted TLS subproblem (\ref{WSTLS}) is 
	less expensive to solve.  Furthermore, $\widetilde{\overrightarrow{\mathcal{B}}}_{ \mathcal{W}}$ acts as an approximation of the least squares solution $\widetilde{\overrightarrow{\mathcal{B}}}_{ \rm ols}$, and its approximate quality depends on the choice of $\{\pi_{i}\}_{i=1}^{n}$.
	Two popular probability distributions are as follows.
	\begin{itemize}
		\item[$\bullet$] \textbf{Uniform sampling.} The sampling probabilities of uniform sampling are de?ned as
		$$\pi_{i}^{Unif}=\frac{1}{n}, \quad\text{for all } i\in[n].$$
		And we refer to the Subsampling TLS method with uniform sampling probability distribution as STLS-Unif.
		\item[$\bullet$] \textbf{Leverage score sampling.}  Let $\mathcal{X}=\mathcal{U}_{\mathcal{X}}*\Sigma_{\mathcal{X}}*\mathcal{V}_{\mathcal{X}}^{T}$ be the thin t-SVD of $\mathcal{X}$. Then the sampling probabilities of leverage score sampling are de?ned as
		$$\pi_{i}^{Lev}=\frac{h_{i}}{\sum_{i=1}^{n}h_{i}}=\frac{\|\mathcal{U}_{\mathcal{X}(i,:,:)}\|_{F}^{2}}{p},\quad\text{for all } i\in[n],$$
		 where $h_{i}=\|\mathcal{U}_{\mathcal{X}(i,:,:)}\|_{F}^{2}$ is referred as the leverage score of the $i$-th sample. And we refer to the Subsampling TLS method with leverage score sampling probability distribution as STLS-Lev.
	 \end{itemize}
 \end{remark}
 \begin{remark}
In the numerical experiments in Section \ref{sub_ex}, we will implement the Subsampling TLS method in the Fourier domain. 
The Fourier version of Algorithm \ref{Subsample_TLS} is given in Algorithms \ref{Subsample_TLS_Four} in the appendix.
  \end{remark}

\section{Optimization perspective}\label{sec_optimization}
We will present the error bounds in the sense of probability for the residual and solution calculated by the Subsampling TLS method, i.e., Algorithm \ref{Subsample_TLS}. Before that, some crucial lemmas are first given as follows.


	\begin{lemma}\label{RanMutil}
	Suppose that $\mathcal{X}\in\mathbb{K}^{m\times n}_{l}$ and $\mathcal{Y}\in\mathbb{K}^{n\times p}_{l}$ are ?xed tubal matrices, and $\mathcal{D}*\mathcal{S}^{T}\in \mathbb{K}^{\tau\times n}_{l}$ is the rescaled  random sampling tubal matrix formed as described in Algorithm \ref{Subsample_TLS}. If the sampling probabilities $\{\pi_{i}\}_{i=1}^{n}$ satisfy
	$$\pi_{i}\geq\beta\frac{\|\mathcal{X}_{(:,i,:)}\|_{F}\|\mathcal{Y}_{(i,:,:)}\|_{F}}{\sum_{i=1}^{n}\|\mathcal{X}_{(:,i,:)}\|_{F}\|\mathcal{Y}_{(i,:,:)}\|_{F}},\quad \beta\in(0,1],$$
	or
	$$\pi_{i}\geq\beta\frac{\|\mathcal{X}_{(:,i,:)}\|_{F}^{2}}{\|\mathcal{X}\|_{F}^{2}},\quad \beta\in(0,1],$$
	then
	\begin{align*}
		\mathbf{E}[\|\mathcal{X}*\mathcal{Y}-\mathcal{X}*\mathcal{S}*\mathcal{D}*\mathcal{D}*\mathcal{S}^{T}*\mathcal{Y}\|_{F}^2]\leq\frac{l}{\beta \tau}\|\mathcal{X}\|_{F}^2\|\mathcal{Y}\|_{F}^2.
	\end{align*}	
\end{lemma}

\begin{remark}\rm
	Theorem 3.1 in \cite{tarzanagh2018fast} also provides a similar result. The key distinction from Lemma \ref{RanMutil} is the way of sampling. Specifically, Theorem 3.1 in \cite{tarzanagh2018fast} is for sampling without replacement, whereas Lemma \ref{RanMutil} is for sampling with replacement.
\end{remark}
\begin{lemma}\label{fullrank}
	For a tubal matrix $\mathcal{X}\in\mathbb{K}_{l}^{n\times p}$ with full tubal rank, i.e., $\text{rank}_{t}(\mathcal{X})=p$, let $\mathcal{X}=\mathcal{U}_{\mathcal{X}}*\Sigma_{\mathcal{X}}*\mathcal{V}_{\mathcal{X}}^{T}$ be its thin t-SVD. Further, let $\Gamma=(\mathcal{D}*\mathcal{S}^{T}*\mathcal{U}_{\mathcal{X}})^{\dag}-(\mathcal{D}*\mathcal{S}^{T}*\mathcal{U}_{\mathcal{X}})^{T}$, where $\mathcal{D}*\mathcal{S}^{T}\in \mathbb{K}^{\tau\times n}_{l}$ is the rescaled  random sampling tubal matrix formed as described in Algorithm \ref{Subsample_TLS}. If the sampling probabilities $\{\pi_{i}\}_{i=1}^{n}$ satisfy
	$$\pi_{i}\geq\beta\frac{\|\mathcal{U}_{\mathcal{X}(i,:,:)}\|_{F}^{2}}{\sum_{i=1}^{n}\|\mathcal{U}_{\mathcal{X}(i,:,:)}\|_{F}^{2}}=\frac{\beta}{p}\|\mathcal{U}_{\mathcal{X}(i,:,:)}\|_{F}^{2},\quad \beta\in(0,1],$$
	and $\tau\geq 40p^{2}l^{2}/\beta\epsilon$ with $\epsilon\in(0,1]$, then with probability at least $0.9$, the following results hold:
	\begin{align}
		&\text{rank}_{t}(\mathcal{D}*\mathcal{S}^{T}*\mathcal{U}_{\mathcal{X}})=\text{rank}_{t}(\mathcal{U}_{\mathcal{X}})=\text{rank}_{t}(\mathcal{X})=p; \label{fullrank_1}\\
		&\|\Gamma\|_{2}^2=\|\Sigma_{\mathcal{D}*\mathcal{S}^{T}*\mathcal{U}_{\mathcal{X}}}^{-1}-\Sigma_{\mathcal{D}*\mathcal{S}^{T}*\mathcal{U}_{\mathcal{X}}}\|_{2}^2;\label{fullrank_2}\\
		&(\mathcal{D}*\mathcal{S}^{T}*\mathcal{X})^{\dag}=\mathcal{V}_{\mathcal{X}}*\Sigma_{\mathcal{X}}^{-1}*(\mathcal{D}*\mathcal{S}^{T}*\mathcal{U}_{\mathcal{X}})^{\dag};\label{fullrank_3}\\
		&\|\Sigma_{\mathcal{D}*\mathcal{S}^{T}*\mathcal{U}_{\mathcal{X}}}^{-1}-\Sigma_{\mathcal{D}*\mathcal{S}^{T}*\mathcal{U}_{\mathcal{X}}}\|_{2}^2\leq\frac{\epsilon}{2},\label{fullrank_4}
	\end{align}	
where $\Sigma_{\mathcal{D}*\mathcal{S}^{T}*\mathcal{U}_{\mathcal{X}}}$ is from the thin t-SVD of $\mathcal{D}*\mathcal{S}^{T}*\mathcal{U}_{\mathcal{X}}=\mathcal{U}_{\mathcal{D}*\mathcal{S}^{T}*\mathcal{U}_{\mathcal{X}}}*\Sigma_{\mathcal{D}*\mathcal{S}^{T}*\mathcal{U}_{\mathcal{X}}}*\mathcal{V}_{\mathcal{D}*\mathcal{S}^{T}*\mathcal{U}_{\mathcal{X}}}^{T}$.
\end{lemma}
\begin{remark}\rm
	Lemma 3.5 in \cite{tarzanagh2018fast} also provides similar results and its key distinction from Lemma \ref{fullrank} is also the way of sampling.
\end{remark}
\begin{lemma}\label{multicom}
	With the same setting as Lemma \ref{fullrank} except that 
	$\tau\geq40pl/\beta\epsilon$, we have, with probability at least $0.9$,
	$$\|\mathcal{U}_{\mathcal{X}}^{T}*\mathcal{S}*\mathcal{D}*\mathcal{D}*\mathcal{S}^{T}*\mathcal{U}_{\mathcal{X}}^{\perp}*{\mathcal{U}_{\mathcal{X}}^{\perp}}^{T}*\overrightarrow{\mathcal{Y}}\|_{F}^2	\leq\frac{\epsilon}{4}\|\mathcal{U}_{\mathcal{X}}^{\perp}*{\mathcal{U}_{\mathcal{X}}^{\perp}}^{T}*\overrightarrow{\mathcal{Y}}\|_{F}^2,$$
	where $\mathcal{U}_{\mathcal{X}}^{\perp}$ is the orthogonal complement of $\mathcal{U}_{\mathcal{X}}$.
\end{lemma}

\begin{lemma}\label{multicomp}
	With the same setting as Lemma \ref{fullrank} except that there is no limit on $\tau$, we have, with probability at least $0.9$:
	$$\|\mathcal{D}*\mathcal{S}^{T}*\mathcal{U}_{\mathcal{X}}^{\perp}*{\mathcal{U}_{\mathcal{X}}^{\perp}}^{T}*\overrightarrow{\mathcal{Y}}\|_{F}^2\leq10\|\mathcal{U}_{\mathcal{X}}^{\perp}*{\mathcal{U}_{\mathcal{X}}^{\perp}}^{T}*\overrightarrow{\mathcal{Y}}\|_{F}^2.$$
\end{lemma}

We now establish the main results.
\begin{theorem}\label{rel_opt}
	Suppose that $\mathcal{X}\in\mathbb{K}_{l}^{n\times p}$ has full tubal rank, i.e., $\text{rank}_{t}(\mathcal{X})=p$, $\overrightarrow{\mathcal{Y}}\in\mathbb{K}_{l}^{n}$, and 
$\widetilde{\overrightarrow{\mathcal{B}}}_{\mathcal{W}}$ is the output calculated by the Subsampling TLS method, i.e., Algorithm \ref{Subsample_TLS}, with sampling probabilities $\{\pi_{i}\}_{i=1}^{n}$  satisfying
	$$\pi_{i}\geq\beta\frac{\|\mathcal{U}_{\mathcal{X}(i,:,:)}\|_{F}^{2}}{\sum_{i=1}^{n}\|\mathcal{U}_{\mathcal{X}(i,:,:)}\|_{F}^{2}}=\frac{\beta}{p}\|\mathcal{U}_{\mathcal{X}(i,:,:)}\|_{F}^{2}, \quad \beta\in(0,1],$$ and the sample size $\tau$ satisfying  $\tau\geq440p^{2}l^{2}/\beta\epsilon$ with $\epsilon\in(0,1]$. Then with probability at least $0.7$:
	\begin{align}
		&f(\widetilde{\overrightarrow{\mathcal{B}}}_{\mathcal{W}})\leq(1+\epsilon)f(\widetilde{\overrightarrow{\mathcal{B}}}_{ols});\label{rele_func}\\
	&	\|\widetilde{\overrightarrow{\mathcal{B}}}_{ols}-\widetilde{\overrightarrow{\mathcal{B}}}_{\mathcal{W}}\|_{F}^2\leq\frac{\epsilon}{\sigma_{\min}^2(\text{\rm bcirc}(\mathcal{X}))}f(\widetilde{\overrightarrow{\mathcal{B}}}_{ols}).\label{rele_beta_func}
	\end{align}	
	If, in addition,  $\|\mathcal{U}_{\mathcal{X}}*\mathcal{U}_{\mathcal{X}}^{T}*\overrightarrow{\mathcal{Y}}\|_{F}^2\geq\gamma\|\overrightarrow{\mathcal{Y}}\|_{F}^2$ for some fixed $\gamma\in(0,1]$, then with probability at least $0.7$:
	\begin{align*}
		\|\widetilde{\overrightarrow{\mathcal{B}}}_{ols}-\widetilde{\overrightarrow{\mathcal{B}}}_{\mathcal{W}}\|_{F}^2\leq\epsilon\kappa^2(\text{\rm bcirc}(\mathcal{X}))(\gamma^{-1}-1)\|\widetilde{\overrightarrow{\mathcal{B}}}_{ols}\|_{F}^{2},
	\end{align*}
 where $\kappa(\text{\rm bcirc}(\mathcal{X}))$ is the condition number of $\text{\rm bcirc}(\mathcal{X})$.
\end{theorem}
  \begin{remark}\label{re_opt}
  	Let $$\mu(\mathcal{U}_{\mathcal{X}})=\frac{nl}{p}\max_{i\in[n]}\|{\mathcal{U}_{\mathcal{X}}}_{(i,:,:)}\|_{F}^{2}$$ be the $\mu$ coherence of $\mathcal{U}_{\mathcal{X}}$, then for all $i\in[n]$,
  	\begin{align*}
    \pi_{i}^{Unif}&=\frac{1}{n}=\frac{l}{\mu(\mathcal{U})p}\max_{i\in[n]}\|\mathcal{U}_{(i,:,:)}\|_{F}^{2}\geq\frac{l}{\mu(\mathcal{U})}\frac{\|\mathcal{U}_{(i,:,:)}\|_{F}^{2}}{p},\\
  	\pi_{i}^{Lev}&=\frac{\|\mathcal{U}_{(i,:,:)}\|_{F}^{2}}{p}.
  	\end{align*}
Thus, both probability distributions satisfy the assumption in Theorem \ref{rel_opt}, and $$\beta^{Unif}\left(=\frac{l}{\mu(\mathcal{U})}\right)<\beta^{Lev}\left(=1\right).$$
Further, considering that the sample size $\tau$ is inversely proportional to $\beta$, we have $\tau^{Unif}>\tau^{Lev}$,  which implies that to achieve the same relative error, the STLS-Unif method requires drawing a larger number of samples than the STLS-Lev method. 
  \end{remark}

\section{Statistical perspective}\label{sec_statistical}

Under the assumption
that $\mathcal{D}*\mathcal{S}^{T}*\mathcal{X}$ is of full tubal rank, we first rewrite $\widetilde{\overrightarrow{\mathcal{B}}}_{\mathcal{W}}$ in (\ref{sol_WTLS}).  Note that $\mathcal{W}$ in (\ref{sol_WTLS}) is an f-diagonal tubal matrix. Then it can be represented as
$$\mathcal{W}=\begin{bmatrix}
	\mathbf{w}_1 & \mathbf{0}	&\cdots & \mathbf{0}\\
	\mathbf{0} & \mathbf{w}_2 &\cdots & \mathbf{0} \\
	\vdots	&\vdots &\ddots &\vdots  \\
	\mathbf{0} & \mathbf{0}&\cdots  &	\mathbf{w}_n \\
\end{bmatrix},$$
where $\mathbf{w}_i$ for $i=1,\cdots,n$ are its diagonal tubal scalars. Let the random weight tubal vector $\overrightarrow{\mathcal{W}} = (\mathbf{w}_1 , \mathbf{w}_2 , \cdots , \mathbf{w}_n )^{ST}$. Thus, $\mathcal{W}=\text{diag}_{t}(\overrightarrow{\mathcal{W}})$, which implies
\begin{align*}
	\widetilde{\overrightarrow{\mathcal{B}}}_{\mathcal{W}}=(\mathcal{X}^{T}*\mathcal{W}*\mathcal{X})^{-1}*\mathcal{X}^{T}*\mathcal{W}*\overrightarrow{\mathcal{Y}}=(\mathcal{X}^{T}*\text{diag}_{t}(\overrightarrow{\mathcal{W}})*\mathcal{X})^{-1}*\mathcal{X}^{T}*\text{diag}_{t}(\overrightarrow{\mathcal{W}})*\overrightarrow{\mathcal{Y}}.
\end{align*}	
Therefore, $	\widetilde{\overrightarrow{\mathcal{B}}}_{\mathcal{W}}$ can be regarded as a function of the random weight tubal vector  $\overrightarrow{\mathcal{W}} $, denoted as 	$\widetilde{\overrightarrow{\mathcal{B}}}_{\mathcal{W}}(\overrightarrow{\mathcal{W}})$.
Now, we give a linear approximation of $\widetilde{\overrightarrow{\mathcal{B}}}_{\mathcal{W}}(\overrightarrow{\mathcal{W}})$ by Taylor series expansion. 
\begin{lemma}\label{taylor_expan}
	Let $\widetilde{\overrightarrow{\mathcal{B}}}_{\mathcal{W}}$ be the output calculated by the Subsampling TLS method, i.e., Algorithm \ref{Subsample_TLS}, and $\overrightarrow{\mathcal{W}}_{0}=\overrightarrow{\mathbf{1}}$ be a tubal vector whose first frontal slice is all ones, and other frontal slices are all zeros. Then, a Taylor expansion of $\widetilde{\overrightarrow{\mathcal{B}}}_{\mathcal{W}}$ around the tubal vector $\overrightarrow{\mathcal{W}}_{0}=\overrightarrow{\mathbf{1}}$ is as follows:
	\begin{align*}
		\widetilde{\overrightarrow{\mathcal{B}}}_{\mathcal{W}}(\overrightarrow{\mathcal{W}})&=\widetilde{\overrightarrow{\mathcal{B}}}_{\mathcal{W}}(\overrightarrow{\mathbf{1}})+(\mathcal{X}^{T}*\mathcal{X})^{-1}*\mathcal{X}^{T}*\text{diag}_{t}(\widetilde{\overrightarrow{\mathcal{E}}})*(\overrightarrow{\mathcal{W}}-\overrightarrow{\mathbf{1}})+\mathcal{R}_{\mathcal{W}}\\
		&=\widetilde{\overrightarrow{\mathcal{B}}}_{\text{ols}}+(\mathcal{X}^{T}*\mathcal{X})^{-1}*\mathcal{X}^{T}*\text{diag}_{t}(\widetilde{\overrightarrow{\mathcal{E}}})*(\overrightarrow{\mathcal{W}}-\overrightarrow{\mathbf{1}})+\mathcal{R}_{\mathcal{W}},
	\end{align*}	
	where $\widetilde{\overrightarrow{\mathcal{E}}}=\overrightarrow{\mathcal{Y}}-\mathcal{X}*\widetilde{\overrightarrow{\mathcal{B}}}_{ols}$  is the TLS residual tubal vector and $\mathcal{R}_{\mathcal{W}}$ is the Taylor expansion remainder.
\end{lemma}

\begin{remark}
	The linear approximation in Lemma \ref{taylor_expan} holds when the remainder $\mathcal{R}_{\mathcal{W}}$ is very small, i.e., $\|\mathcal{R}_{\mathcal{W}}\|_{F}=o_{p}(\|\overrightarrow{\mathcal{W}}-\overrightarrow{\mathbf{1}}\|_{F})$, where $o_{p}$ means "little o" with high probability over the randomness in the random tubal vector $\overrightarrow{\mathcal{W}}$. For $\mathcal{R}_{\mathcal{W}}$, it depends strongly on the sampling process, and is different for different sampling probability distributions. In theory, we cannot yet state when $\mathcal{R}_{\mathcal{W}}$ is small.  In Section \ref{sub_ex}, we will numerically demonstrate the validity of the linear approximation.
\end{remark}

We next present the expressions of the conditional and unconditional expectations and variances for the estimator $\widetilde{\overrightarrow{\mathcal{B}}}_{\mathcal{W}}$. 

\begin{theorem}\label{exvar_sta}
	Let $\widetilde{\overrightarrow{\mathcal{B}}}_{\mathcal{W}}$ be the output calculated by the Subsampling TLS method, i.e., Algorithm \ref{Subsample_TLS}. Then the conditional expectation and variance of $\widetilde{\overrightarrow{\mathcal{B}}}_{\mathcal{W}}$ 
	are given by:
\begin{align}
	\mathbf{E}_{\overrightarrow{\mathcal{W}}}[\widetilde{\overrightarrow{\mathcal{B}}}_{\mathcal{W}}\mid\overrightarrow{\mathcal{Y}}]=&\widetilde{\overrightarrow{\mathcal{B}}}_{\text{ols}}+\mathbf{E}_{\overrightarrow{\mathcal{W}}}[\mathcal{R}_{\mathcal{W}}];\label{unconditionexpe}\\
	\mathbf{Var}_{\overrightarrow{\mathcal{W}}}[\widetilde{\overrightarrow{\mathcal{B}}}_{\mathcal{W}}\mid\overrightarrow{\mathcal{Y}}]=&(\mathcal{X}^{T}*\mathcal{X})^{-1}*\mathcal{X}^{T}*\text{diag}_{t}(\widetilde{\overrightarrow{\mathcal{E}}})*\text{diag}_{t}\left(\frac{1}{\tau\pi_{i}}\mathbf{1}\right)*\text{diag}_{t}(\widetilde{\overrightarrow{\mathcal{E}}})*\mathcal{X}*(\mathcal{X}^{T}*\mathcal{X})^{-1}\nonumber\\
	&+\mathbf{Var}_{\overrightarrow{\mathcal{W}}}[\mathcal{R}_{\mathcal{W}}],\label{unconditionvar}
\end{align}	
where $\mathbf{1}$ is the tubal scalar with components $\mathbf{1}_{(1)}=1$ and $\mathbf{1}_{(k)}=0$ for $k=2,3,\cdots,l$ and $\text{diag}_{t}\left(\frac{1}{\tau\pi_{i}}\mathbf{1}\right)$ is an f-diagonal tubal matrix with diagonal tubal scalars $\left\{\frac{1}{\tau\pi_{i}}\mathbf{1}\right\}_{i=1}^{n}$.

Further, the unconditional expectation and variance of $\widetilde{\overrightarrow{\mathcal{B}}}_{\mathcal{W}}$ are:
\begin{align}
	\mathbf{E}[\widetilde{\overrightarrow{\mathcal{B}}}_{\mathcal{W}}]=&\overrightarrow{\mathcal{B}}_{0}+\mathbf{E}[\mathcal{R}_{\mathcal{W}}];\label{conditionexpe}\\
	\mathbf{Var}[\widetilde{\overrightarrow{\mathcal{B}}}_{\mathcal{W}}]=&\sigma^{2}(\mathcal{X}^{T}*\mathcal{X})^{-1}+\frac{\sigma^{2}}{\tau}(\mathcal{X}^{T}*\mathcal{X})^{-1}*\mathcal{X}^{T}*\text{diag}_{\text{t}}\left(\frac{1}{\pi_{i}}(\mathbf{1}-	\mathcal{X}_{(i,:,:)}*(\mathcal{X}^{T}*\mathcal{X})^{-1}*\mathcal{X}_{(i,:,:)}^{T})\right)\nonumber\\
	&*\mathcal{X}*(\mathcal{X}^{T}*\mathcal{X})^{-1}+\mathbf{Var}[\mathcal{R}_{\mathcal{W}}].\label{conditionvar}
\end{align}	
\end{theorem}	

\begin{remark}
	The equation (\ref{unconditionexpe}) states that,  when $\overrightarrow{\mathcal{Y}}$ is given and the $\mathbf{E}_{\overrightarrow{\mathcal{W}}}[\mathcal{R}_{\mathcal{W}}]$ term is negligible, the estimator $\widetilde{\overrightarrow{\mathcal{B}}}_{\mathcal{W}}$ is approximately unbiased relative to the least squares estimator $\widetilde{\overrightarrow{\mathcal{B}}}_{\text{ols}}$, while the equation (\ref{conditionexpe})  states that, when the $\mathbf{E}[\mathcal{R}_{\mathcal{W}}]$ term is negligible, the estimator $\widetilde{\overrightarrow{\mathcal{B}}}_{\mathcal{W}}$ is approximately unbiased relative to the true solution $\overrightarrow{\mathcal{B}}_{0}$.
\end{remark}
\begin{remark}
	From the equations (\ref{unconditionvar}) and (\ref{conditionvar}), we know that both the conditional variance and the (second term of the) unconditional variance are inversely proportional to the sample size $\tau$, and both contain a sandwich-type expression, the middle of which depends on the sampling probabilities $\{\pi_{i}\}_{i=1}^{n}$. Moreover, the first term of $\mathbf{Var}[\widetilde{\overrightarrow{\mathcal{B}}}_{\mathcal{W}}]$ is $\sigma^{2}(\mathcal{X}^{T}*\mathcal{X})^{-1}$, which equals the variance of $\widetilde{\overrightarrow{\mathcal{B}}}_{\text{ols}}$. This implies that $\mathbf{Var}[\widetilde{\overrightarrow{\mathcal{B}}}_{\mathcal{W}}]>\mathbf{Var}[\widetilde{\overrightarrow{\mathcal{B}}}_{\text{ols}}]$. The conclusion is consistent with the one in \cite{jin2017generalized} that $\widetilde{\overrightarrow{\mathcal{B}}}_{\rm ols}$ is the best linear unbiased estimator, namely, its  variance is minimal among all linear unbiased estimators.
\end{remark}

In the following, we provide two corollaries by specializing Theorem \ref{exvar_sta}, one for uniform sampling and the other for leverage score sampling.
\begin{corollary}\label{exvar_unif}
	Let $\widetilde{\overrightarrow{\mathcal{B}}}_{\text{Unif}}$ be the output calculated by the STLS-Unif method. Then the conditional expectation and variance of $\widetilde{\overrightarrow{\mathcal{B}}}_{\text{Unif}}$ 
	are given by:
	\begin{align*}
		\mathbf{E}_{\overrightarrow{\mathcal{W}}}[\widetilde{\overrightarrow{\mathcal{B}}}_{\text{Unif}}\mid\overrightarrow{\mathcal{Y}}]=&\widetilde{\overrightarrow{\mathcal{B}}}_{\text{ols}}+\mathbf{E}_{\overrightarrow{\mathcal{W}}}[\mathcal{R}_{\text{Unif}}];\\
		\mathbf{Var}_{\overrightarrow{\mathcal{W}}}[\widetilde{\overrightarrow{\mathcal{B}}}_{\text{Unif}}\mid\overrightarrow{\mathcal{Y}}]=&\frac{n}{\tau}(\mathcal{X}^{T}*\mathcal{X})^{-1}*\mathcal{X}^{T}*\text{diag}_{t}(\widetilde{\overrightarrow{\mathcal{E}}})*\text{diag}_{t}(\widetilde{\overrightarrow{\mathcal{E}}})*\mathcal{X}*(\mathcal{X}^{T}*\mathcal{X})^{-1}+\mathbf{Var}_{\overrightarrow{\mathcal{W}}}[\mathcal{R}_{\text{Unif}}].
	\end{align*}	
	
	Further, the unconditional expectation and variance of $\widetilde{\overrightarrow{\mathcal{B}}}_{\text{Unif}}$ are:
	\begin{align*}
		\mathbf{E}[\widetilde{\overrightarrow{\mathcal{B}}}_{\text{Unif}}]=&\overrightarrow{\mathcal{B}}_{0}+\mathbf{E}[\mathcal{R}_{\text{Unif}}];\\
		\mathbf{Var}[\widetilde{\overrightarrow{\mathcal{B}}}_{\text{Unif}}]=&\sigma^{2}(\mathcal{X}^{T}*\mathcal{X})^{-1}+\frac{n\sigma^{2}}{\tau}(\mathcal{X}^{T}*\mathcal{X})^{-1}*\mathcal{X}^{T}*\text{diag}_{\text{t}}\left(\mathbf{1}-	\mathcal{X}_{(i,:,:)}*(\mathcal{X}^{T}*\mathcal{X})^{-1}*\mathcal{X}_{(i,:,:)}^{T}\right)\\
		&*\mathcal{X}*(\mathcal{X}^{T}*\mathcal{X})^{-1}+\mathbf{Var}[\mathcal{R}_{\text{Unif}}].
	\end{align*}	
\end{corollary}	
\begin{corollary}\label{exvar_lev}
	Let $\widetilde{\overrightarrow{\mathcal{B}}}_{\text{Lev}}$ be the output calculated by the STLS-Lev method. Then the conditional expectation and variance of $\widetilde{\overrightarrow{\mathcal{B}}}_{\text{Lev}}$ 
	are given by:
	\begin{align*}
		\mathbf{E}_{\overrightarrow{\mathcal{W}}}[\widetilde{\overrightarrow{\mathcal{B}}}_{\text{Lev}}\mid\overrightarrow{\mathcal{Y}}]=&\widetilde{\overrightarrow{\mathcal{B}}}_{\text{ols}}+\mathbf{E}_{\overrightarrow{\mathcal{W}}}[\mathcal{R}_{\text{Lev}}];\\
		\mathbf{Var}_{\overrightarrow{\mathcal{W}}}[\widetilde{\overrightarrow{\mathcal{B}}}_{\text{Lev}}\mid\overrightarrow{\mathcal{Y}}]=&\frac{p}{\tau}(\mathcal{X}^{T}*\mathcal{X})^{-1}*\mathcal{X}^{T}*\text{diag}_{t}(\widetilde{\overrightarrow{\mathcal{E}}})*\text{diag}_{\text{t}}\left(\frac{1}{h_{i}}\mathbf{1}\right)*\text{diag}_{t}(\widetilde{\overrightarrow{\mathcal{E}}})*\mathcal{X}*(\mathcal{X}^{T}*\mathcal{X})^{-1}\\
		&+\mathbf{Var}_{\overrightarrow{\mathcal{W}}}[\mathcal{R}_{\text{Lev}}].
	\end{align*}	
	
	Further,
	the unconditional expectation and variance of $\widetilde{\overrightarrow{\mathcal{B}}}_{\text{Lev}}$ are:
	\begin{align*}
		\mathbf{E}[\widetilde{\overrightarrow{\mathcal{B}}}_{\text{Lev}}]=&\overrightarrow{\mathcal{B}}_{0}+\mathbf{E}[\mathcal{R}_{\text{Lev}}];\\
		\mathbf{Var}[\widetilde{\overrightarrow{\mathcal{B}}}_{\text{Lev}}]=&\sigma^{2}(\mathcal{X}^{T}*\mathcal{X})^{-1}+\frac{p\sigma^{2}}{\tau}(\mathcal{X}^{T}*\mathcal{X})^{-1}*\mathcal{X}^{T}*\text{diag}_{\text{t}}\left(\frac{1}{h_{i}}(\mathbf{1}-	\mathcal{X}_{(i,:,:)}*(\mathcal{X}^{T}*\mathcal{X})^{-1}*\mathcal{X}_{(i,:,:)}^{T})\right)\\
		&*\mathcal{X}*(\mathcal{X}^{T}*\mathcal{X})^{-1}+\mathbf{Var}[\mathcal{R}_{\text{Lev}}].
	\end{align*}	
\end{corollary}	
\begin{remark}\label{stat_uniflev}
	 From Corollaries \ref{exvar_unif} and \ref{exvar_lev}, on the one hand, we find that $\mathbf{Var}[\widetilde{\overrightarrow{\mathcal{B}}}_{\text{Unif}}]$ and $\mathbf{Var}[\widetilde{\overrightarrow{\mathcal{B}}}_{\text{Lev}}]$ depend on $n/\tau$ and $p/\tau$, respectively, thus $\mathbf{Var}[\widetilde{\overrightarrow{\mathcal{B}}}_{\text{Unif}}]$ can be very large unless $\tau\approx n$ and $\mathbf{Var}[\widetilde{\overrightarrow{\mathcal{B}}}_{\text{Lev}}]$ can be very small if $p\ll \tau\ll n$. On the other hand, the sandwich-type expression of $\mathbf{Var}[\widetilde{\overrightarrow{\mathcal{B}}}_{\text{Lev}}]$ can be inflated by very small leverage scores, while that of $\mathbf{Var}[\widetilde{\overrightarrow{\mathcal{B}}}_{\text{Unif}}]$ cannot.
\end{remark}

It should be worth noting that Remark \ref{re_opt} shows that the STLS-Lev method outperforms the STLS-Unif method from the optimization perspective. However, Remark \ref{stat_uniflev} indicates that from the statistical perspective, neither method is dominant and their variances may both be large. With this in mind, we next give an optimal subsampling probability distribution.

From Theorem \ref{exvar_sta}, we know that when the $\mathbf{E}[\mathcal{R}_{\mathcal{W}}]$ term in (\ref{conditionexpe})  and the $\mathbf{Var}[\mathcal{R}_{\mathcal{W}}]$  term in (\ref{conditionvar}) are negligible, the estimator $\widetilde{\overrightarrow{\mathcal{B}}}_{\mathcal{W}}$ is approximately unbiased relative to $\overrightarrow{\mathcal{B}}_{0}$ and its variance can be approximated as follows:
\begin{align*}
	\mathbf{Var}[\widetilde{\overrightarrow{\mathcal{B}}}_{\mathcal{W}}]\approx&\sigma^{2}(\mathcal{X}^{T}*\mathcal{X})^{-1}+\frac{\sigma^{2}}{\tau}(\mathcal{X}^{T}*\mathcal{X})^{-1}*\mathcal{X}^{T}*\text{diag}_{\text{t}}\left(\frac{1}{\pi_{i}}(\mathbf{1}-	\mathcal{X}_{(i,:,:)}*(\mathcal{X}^{T}*\mathcal{X})^{-1}*\mathcal{X}_{(i,:,:)}^{T})\right)\nonumber\\
	&*\mathcal{X}*(\mathcal{X}^{T}*\mathcal{X})^{-1}.
\end{align*}	
Thus, we can make the mean squared error to attain its minimum approximately by minimizing $\text{trace}(\mathbf{Var}[\widetilde{\overrightarrow{\mathcal{B}}}_{\mathcal{W}}])$, where $\text{trace}(\mathcal{ A})$ means the trace of $\mathcal{ A}$ and is defined as $\text{trace}(\mathcal{ A})=\frac{1}{l}\sum_{k=1}^{l}\text{trace}(\widehat{\mathcal{ A}}_{(k)})$.
\begin{theorem}\label{optimalprob}
When the sampling probabilities are chosen as
$$\pi_{i}^{Opt}=\frac{\sqrt{\frac{1}{l}\sum_{k=1}^{l}\left(1-\|\widehat{\mathcal{U}_{\mathcal{X}}}_{(i,:,k)}\|_{2}^{2}\right)\|\widehat{\mathcal{ X}}_{(i,:,k)}\|_{2}^2}}{\sum_{i=1}^{n}\sqrt{\frac{1}{l}\sum_{k=1}^{l}\left(1-\|\widehat{\mathcal{U}_{\mathcal{X}}}_{(i,:,k)}\|_{2}^{2}\right)\|\widehat{\mathcal{ X}}_{(i,:,k)}\|_{2}^2}},\quad\text{for all } i\in[n],$$
then $\text{trace}(\mathbf{Var}[\widetilde{\overrightarrow{\mathcal{B}}}_{\mathcal{W}}])$ attains its minimum.
\end{theorem}
\begin{remark}
We refer to the Subsampling TLS method with the optimal subsampling probability distribution as STLS-Opt, and let $\widetilde{\overrightarrow{\mathcal{B}}}_{\text{Opt}}$ be the output calculated by the STLS-Opt method.
\end{remark}

 Although $\mathbf{Var}[\widetilde{\overrightarrow{\mathcal{B}}}_{\text{Opt}}]$ is smaller than $\mathbf{Var}[\widetilde{\overrightarrow{\mathcal{B}}}_{\text{Unif}}]$ and $\mathbf{Var}[\widetilde{\overrightarrow{\mathcal{B}}}_{\text{Lev}}]$, the form of $\pi_{i}^{\text{Opt}}$ is a little more complicated than that of $\pi_{i}^{\text{Unif}}$ and $\pi_{i}^{\text{Lev}}$. As done in \cite{ma2014statistical}, we next present another formally simpler probability distribution to address the issue that $\mathbf{Var}[\widetilde{\overrightarrow{\mathcal{B}}}_{\text{Unif}}]$ and $\mathbf{Var}[\widetilde{\overrightarrow{\mathcal{B}}}_{\text{Lev}}]$ may be too large. Specifically, the sampling probabilities are de?ned based on $\pi_{i}^{Unif}$ and $\pi_{i}^{Lev}$ as
		 $$\pi_{i}^{Slev}= \alpha\pi_{i}^{Lev}+(1-\alpha)\pi_{i}^{Unif}, \quad\text{for all } i\in[n],$$
		 where $\alpha\in(0,1)$. We call it the shrinked leverage score sampling probability distribution, and refer to the Subsampling TLS method with such  probability distribution as STLS-Slev.
\begin{remark}\label{opt_Slev}
Since for all $i\in[n]$,
\begin{align*}
  	\pi_{i}^{Slev}&=\alpha\pi_{i}^{Lev}+(1-\alpha)\pi_{i}^{Unif}
 \geq\left(\alpha+\frac{(1-\alpha)l}{\mu(\mathcal{U})}\right)\frac{\|\mathcal{U}_{(i,:,:)}\|_{F}^{2}}{p},
  	\end{align*}
   this probability distribution also satisfies the assumption in Theorem \ref{rel_opt}. Combining with Remark \ref{re_opt}, we have
   that $\tau^{Lev}<\tau^{Slev}<\tau^{Unif}$. In addition, when $1-\alpha$ is not too large, e.g., $1-\alpha=0.1$, $\tau^{Slev}$ is only a little larger than $\tau^{Lev}$.
   \end{remark}
  \begin{remark}\label{sta_Slev}
  Let $\widetilde{\overrightarrow{\mathcal{B}}}_{\text{Slev}}$ be the output calculated by the STLS-Slev method.
  According to Lemma \ref{exvar_sta}, we can also get the the conditional and unconditional expectations and variances of $\widetilde{\overrightarrow{\mathcal{B}}}_{\text{Slev}}$. However, they are much messier and less transparent than those of $\widetilde{\overrightarrow{\mathcal{B}}}_{\text{Unif}}$ and $\widetilde{\overrightarrow{\mathcal{B}}}_{\text{Lev}}$, so we will not give them explicitly and just highlight a few key points. First, when $1-\alpha$ is not extremely small, e.g., $1-\alpha=0.1$, then none of the shrinked leverage score sampling probabilities is too small, and thus $\mathbf{Var}[\widetilde{\overrightarrow{\mathcal{B}}}_{\text{Slev}}]$  will not be inflated too much. Second, when $1-\alpha$ is not too large, e.g. $1-\alpha=0.1$, then $\mathbf{Var}[\widetilde{\overrightarrow{\mathcal{B}}}_{\text{Slev}}]$ has a scale of $p/\tau$ rather than $n/\tau$. Consequently, $\widetilde{\overrightarrow{\mathcal{B}}}_{\text{Slev}}$ benefits from both $\widetilde{\overrightarrow{\mathcal{B}}}_{\text{Unif}}$ and $\widetilde{\overrightarrow{\mathcal{B}}}_{\text{Lev}}$.
\end{remark}

 \section{Numerical experiments}\label{sub_ex}
In this section, we use numerical experiments to test our method and verify our theoretical results. All experiments are carried out by using MATLAB R2021b on a standard MacBook Pro 2021 with an Apple M1 pro processor and 16GB memory.
\subsection{Simulation data}
We generate the data for the model (\ref{linearmodel}) with $n=5000$, $p=10$, and $l=10$ in the following way. The elements of $\overrightarrow{\mathcal{E}}$ are drawn i.i.d. from $N(0, 9)$, and the frontal slices of $\mathcal{ X}$, i.e., $\mathcal{ X}_{(k)}$ for all $k\in[l]$, are generated from the following distributions:
 \begin{itemize}
	\item Multivariate normal distribution $\text{N}(\mathbf{1}_{p},\Sigma)$, where $\mathbf{1}_{p}$ is a $p\times 1$ column vector of ones and the $(i,j)$-th element of $\Sigma$ is set to $2\times 0.5^{\vert i-j\vert}$, for $i,j=1,\cdots,p$. We refer to this as MN data.
	\item Multivariate t-distribution with $3$ degrees of freedom and the covariance matrix $\Sigma$ mentioned above. We refer to this as T3 data.
	\item Multivariate t-distribution with $1$ degrees of freedom and the covariance matrix $\Sigma$ mentioned above. We refer to this as T1 data.
\end{itemize}
And, we take ${\overrightarrow{\mathcal{B}}_{0}}_{(k)}=(\mathbf{1}_{2}, 0.1\times\mathbf{1}_{p-4}, \mathbf{1}_{2})$ for all $k\in[l]$ and set $\overrightarrow{\mathcal{Y}}=\mathcal{X}*\overrightarrow{\mathcal{B}}_{0}+\overrightarrow{\mathcal{E}}$.

In our experiments, we use the optimization and statistical criteria to evaluate each estimator. Speci?cally, let $\widetilde{\overrightarrow{\mathcal{B}}}_{b}$ be the estimator in the $b$-th replicated sample and $\text{SM}(\widetilde{\overrightarrow{\mathcal{B}}})=\frac{1}{\text{500}}\sum_{\text{b}=1}^{\text{500}}\widetilde{\overrightarrow{\mathcal{B}}}_{b}$ be the sample mean of $\widetilde{\overrightarrow{\mathcal{B}}}_{b}$ across $500$ replicated samples.

 \begin{itemize}
	\item Optimization criteria include sample mean relative function value (SMRFV) and sample mean relative error (SMRE), that is,
  $$\text{SMRFV}(\widetilde{\overrightarrow{\mathcal{B}}})=\frac{1}{500}\sum_{b=1}^{500} \frac{\lvert f(\widetilde{\overrightarrow{\mathcal{B}}}_{b})-f(\widetilde{\overrightarrow{\mathcal{B}}}_{\text{ols}})\rvert}{f(\widetilde{\overrightarrow{\mathcal{B}}}_{\text{ols}})};\quad \text{SMRE}(\widetilde{\overrightarrow{\mathcal{B}}})=\frac{1}{500}\sum_{b=1}^{500}\frac{\|\widetilde{\overrightarrow{\mathcal{B}}}_{b}-\widetilde{\overrightarrow{\mathcal{B}}}_{\text{ols}}\|_{F}^2}{\|\widetilde{\overrightarrow{\mathcal{B}}}_{\text{ols}}|_{F}^2}.$$
\item Statistical criteria include sample squared bias (SSB), sample variance (SV) and sample mean square error (SMSE), that is,
$$\text{SSB}(\widetilde{\overrightarrow{\mathcal{B}}})=\|\text{SM}(\widetilde{\overrightarrow{\mathcal{B}}})-\overrightarrow{\mathcal{B}}_{0}\|_{F}^{2};\quad \text{SV}(\widetilde{\overrightarrow{\mathcal{B}}})=\frac{1}{500}\sum_{b=1}^{500}\|\widetilde{\overrightarrow{\mathcal{B}}}_{b}-\text{SM}(\widetilde{\overrightarrow{\mathcal{B}}})\|_{F}^{2};$$
$$\text{SMSE}(\widetilde{\overrightarrow{\mathcal{B}}})=\frac{1}{500}\sum_{b=1}^{500}\|\widetilde{\overrightarrow{\mathcal{B}}}_{b}-\overrightarrow{\mathcal{B}}_{0}\|_{F}^{2}.$$
\end{itemize}

 \begin{example}\label{ex1}\rm
 Consider that the TLS (\ref{TLS}) problem can be transformed into the following MLS problem:
 	\begin{align}\label{MLS}
 		\text{bcirc}(\mathcal{ X})\text{unfold}(\overrightarrow{\mathcal{B}})=\text{unfold}(\overrightarrow{\mathcal{Y}}),
 	\end{align}
 	and it can be solved by the Subsampling MLS method. In this example, we compare the empirical performance of the Subsampling TLS method for (\ref{TLS}) and the Subsampling MLS method for (\ref{MLS}). Here, we only consider uniform sampling and leverage score sampling. Specifically, we compare the performance of the  STLS-Unif and STLS-Lev methods with sample size  $\tau=200:100:1000$ and the SMLS-Unif and SMLS-Lev methods with sample size $\tau$ or $l\tau=10\tau$. The numerical results are provided in Figures \ref{Unif_TLS_MLS} and \ref{Lev_TLS_MLS}, from which we can see that for MN, T3 and T1 data, the STLS-Unif and STLS-Lev methods perform much better than the corresponding SMLS-Unif and SMLS-Lev methods 
 	for the same sample size. If the sample size of the SMLS-Unif and SMLS-Lev methods is set to be $l=10$ times that of the STLS-Unif and STLS-Lev methods, they perform almost the same as expected. However, the STLS-Unif and STLS-Lev methods require much less time. Therefore, the Subsampling TLS method for (\ref{TLS}) indeed has more advantages compared with 
 	the Subsampling MLS method for (\ref{MLS}).
\begin{figure}[htbp]
	\centering
	\includegraphics[width=0.85 \textwidth ]{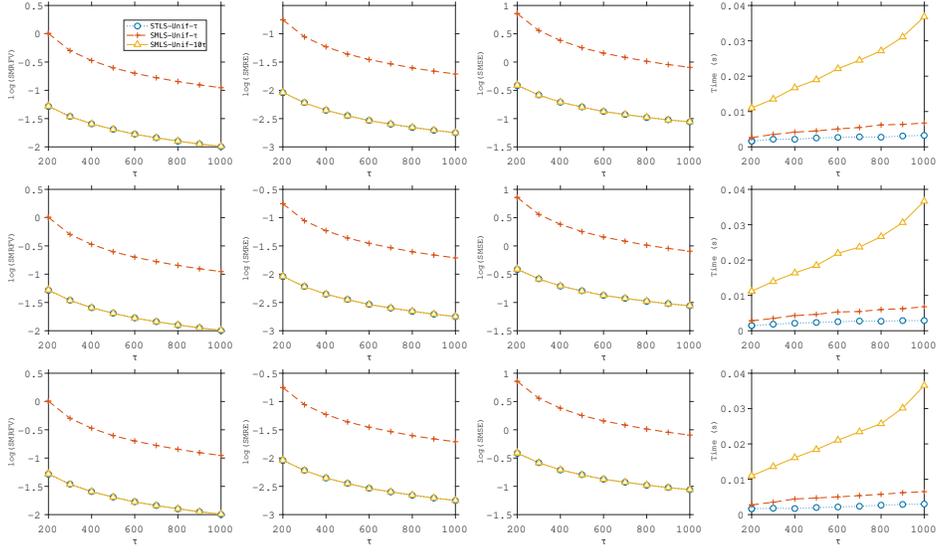}
	\caption{Comparison of the performance of the STLS-Unif method with sample size $\tau=200:100:1000$ (STLS-Unif-$\tau$) and the SMLS-Unif method with sample size $\tau$ (SMLS-Unif-$\tau$) or $10\tau$ (SMLS-Unif-$10\tau$). The three row panels correspond to MN, T3, and T1 data, respectively.  }\label{Unif_TLS_MLS}
\end{figure}
\begin{figure}[htbp]
	\centering
	\includegraphics[width=0.85 \textwidth]{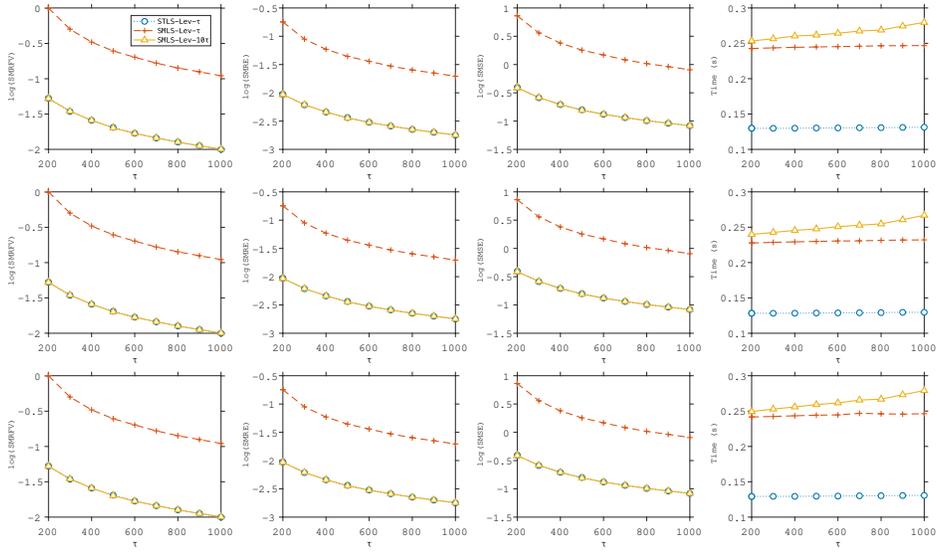}
	\caption{Comparison of the performance of the STLS-Lev method with sample size $\tau=200:100:1000$ (STLS-Lev-$\tau$) and the SMLS-Lev method with sample size $\tau$ (SMLS-Lev-$\tau$)  or $10\tau$ (SMLS-Lev-$10\tau$). The three row panels correspond to MN, T3, and T1 data, respectively.}\label{Lev_TLS_MLS}
\end{figure}
 \end{example}

 \begin{example}\rm
From the definition of the shrinked leverage score sampling probabilities, we know that the parameter $\alpha$ a?ects the behavior of the STLS-Slev method to some extent. With this in mind, we demonstrate in this example how the specific effect of the parameter $\alpha$ is when the STLS-Slev method is used for different data. We set the sample size $\tau=3p, 5p, 10p$ and $\alpha=0:0.1:1$.  From Figure \ref{alpha}, the following can be seen. First, for MN and T3 data, $\alpha$ has an irregular impact on the SSB value of the STLS-Slev method. However, these values are much smaller compared with the corresponding ones of SV. So the final impact on the STLS-Slev method is limited. Second, for MN data, the SMRFV, SMRE, SV and SMSE values of the STLS-Slev method with different $\alpha$ are almost the same, and for T3 data, $\alpha$ has very subtle influence on the SMRFV, SMRE, SV and SMSE values.
Third, for T1 data, these curves are U-shaped, which indicates that the STLS-Slev method first improves as $\alpha$ increases, but gradually degrades as $\alpha$ keeps growing. 
That is, $\alpha$ cannot be set too large or too small, which is consistent with Remarks \ref{opt_Slev} and \ref{sta_Slev}. Therefore, we will set $\alpha=0.9$ in the following  examples.

\begin{figure}[htbp]
	\centering
	\includegraphics[width=0.9 \textwidth]{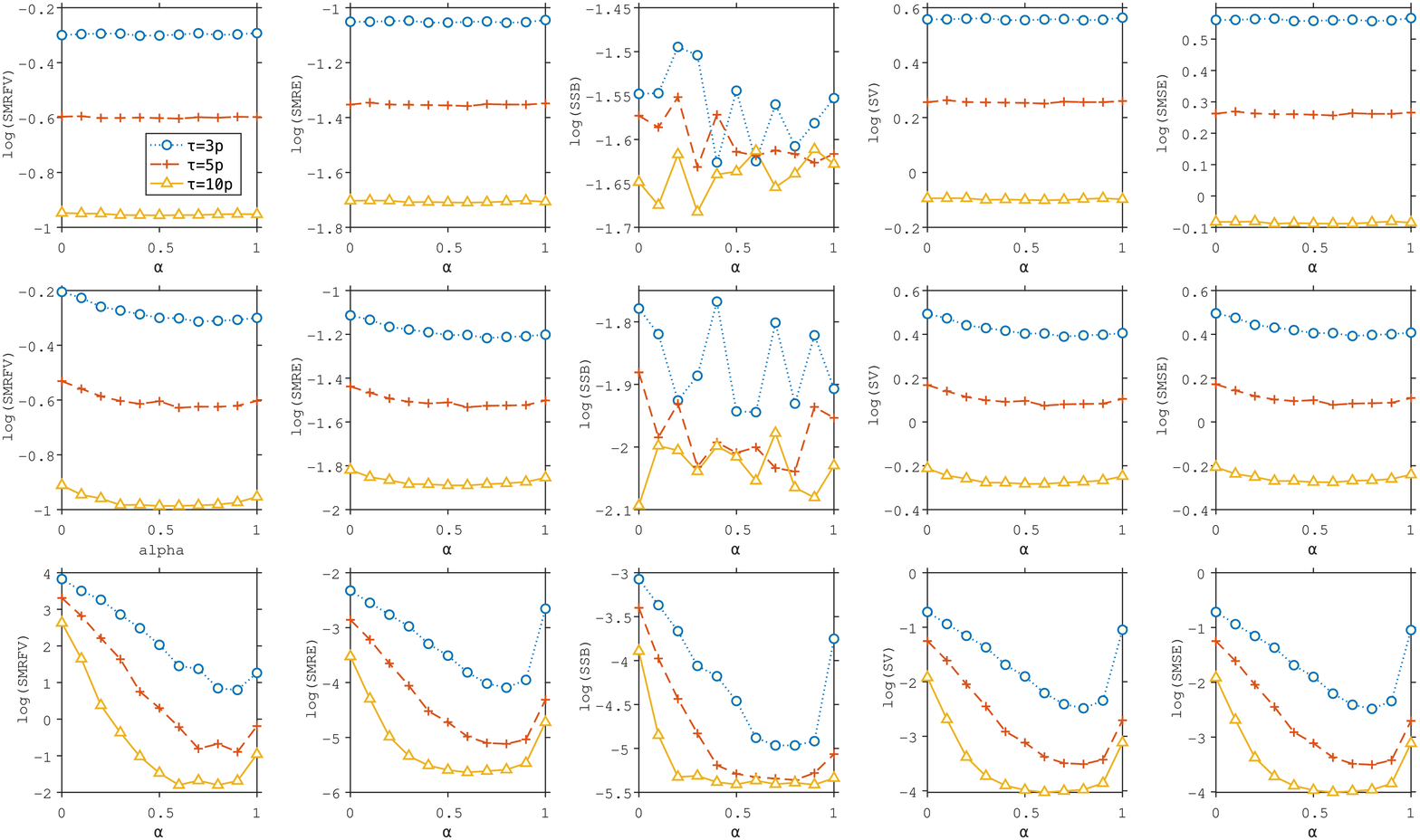}
	\caption{The effect of $\alpha$ in the STLS-Slev method. The three row panels correspond to MN, T3, and T1 data, respectively. }\label{alpha}
\end{figure}
\end{example}

 \begin{example}\label{ex3}\rm
 	In this example, we  compare the empirical performance of the Subsampling TLS method with different sampling probability distributions, including the STLS-Unif, STLS-Lev, STLS-Slev and STLS-Opt methods. We set the sample size $\tau=200:100:1000$. According to Figure \ref{compare_method}, there are numerous points worth noting. First, in terms of the SMRFV, SMRE, SV and SMSE values, the four methods is very similar for MN data, less similar for T3 data, and very different for T1 data. This is mainly because the leverage scores are very uniform for MN data and very nonuniform for T1 data. Furthermore, when they are different, the STLS-Opt method performs the best, and the STLS-Slev method is  superior to the STLS-Unif and STLS-Lev methods, which is consistent with the fact that the STLS-Slev method benefits from both STLS-Unif and STLS-Lev methods. Second, as the sample size $\tau$ increases, the SMRFV, SMRE, SV and SMSE values of the four methods all gradually decrease, which veri?es the fact that these values are all inversely proportional to the sample size $\tau$. Third, SV obviously dominates SSB, which is consistent with the fact that the estimator calculated by the Subsampling TLS method is approximately unbiased and the squared bias is proved to be negligible with respective to the variance.
 	\begin{figure}[htbp]
 		\centering
 		\includegraphics[width=0.9 \textwidth]{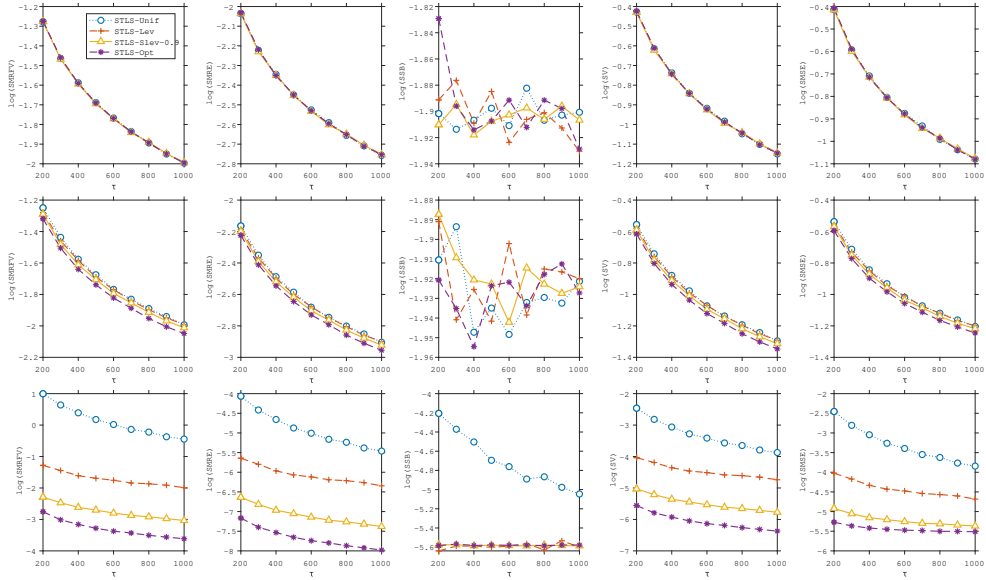}
 		\caption{Comparison of the performance of the STLS-Unif, STLS-Lev, STLS-Slev and STLS-Opt methods on simulation data. The three row panels correspond to MN, T3, and T1 data, respectively.}\label{compare_method}
 	\end{figure}

\end{example}
\subsection{Real data}
 \begin{example} \rm
 In this example, we evaluate the performance of the STLS-Unif, STLS-Lev, STLS-Slev and STLS-Opt methods on the Air Quality data set \cite{AirQuality}. Our primary interest here is to analyze the relationships between some air pollutants. Specifically, we ?t model (\ref{linearmodel}) with the true hourly averaged Benzene concentration as the response and the true hourly averaged NO and $\text{NO}^2$ concentrations as the predictors. Since each of these three has $9357$ hours of recordings, we may obtain three tubal vectors of size $1559\times1\times6$ by treating the six consecutive hours of recordings as a tubal scalar of length $6$ and removing the last three recordings. We randomly divide the obtained $1559$ samples into $1403$ for training and $156$ for testing and set the sample size $\tau=100:50:500$. All data is standardized before being fed into the model. Since the true solution $\overrightarrow{\mathcal{B}}_{0}$ for the real dataset is unknown, we use the least squares estimator $\widetilde{\overrightarrow{\mathcal{B}}}_{\text{ols}}$ in the statistical criteria instead. Inspecting in  Figure \ref{air}, we get similar observations to Example \ref{ex3}. First, in terms of the SMRFV, SMRE, SV and SMSE values, the STLS-Opt method is optimal and the STLS-Slev method is better than the STLS-Unif and STLS-Lev methods. Second, the SMRFV, SMRE, SV and SMSE values (both estimations and predictions) of the four methods gradually decrease as the $\tau$ increases. Third, the bias e?ect is small enough to be ignored compared to the variance.
 	 	\begin{figure}[htbp]
 		\centering
 		\includegraphics[ width=0.9 \textwidth]{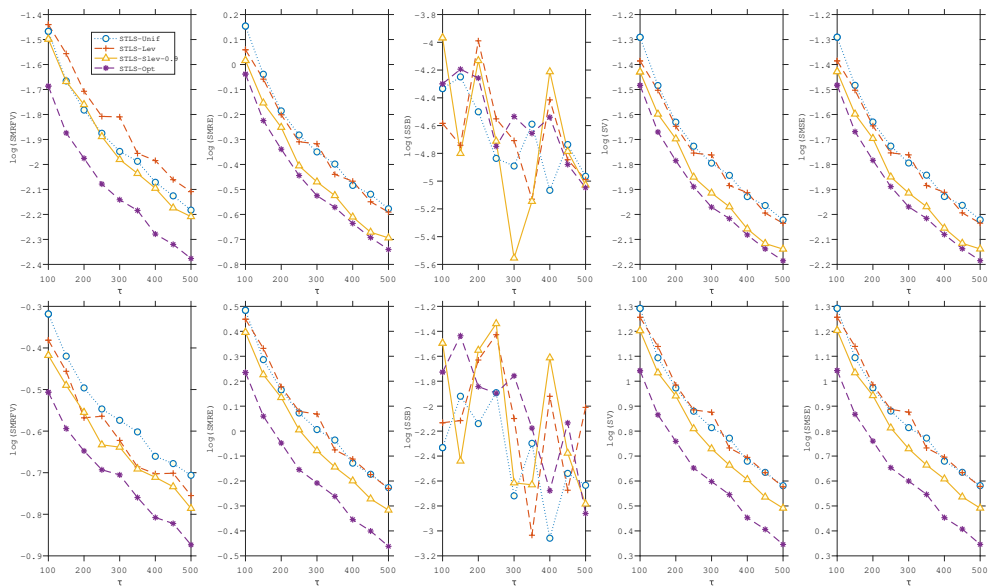}
 		\caption{Comparison of the performance of the STLS-Unif, STLS-Lev, STLS-Slev and STLS-Opt methods on the Air Quality data set. The two row panels correspond to the estimations and predictions, respectively.}\label{air}
 	\end{figure}
 \end{example}

\section{Conclution}
	In this work, we propose the Subsampling TLS method for solving tensor least squares. In theory, we analyze the quality of the proposed method from both optimization and statistical perspectives in details. Numerical results verify the feasibility and effectiveness of the proposed method 
 and the correctness of the theoretical results.
\bibliography{mybibfile}

\appendix
\section{Proofs for theoretical results\label{app1}}
\begin{proof}[Proof of Lemma~{\upshape\ref{RanMutil}}] \rm
	Let $\mathcal{M}=\mathcal{D}*\mathcal{S}^{T}$. Fix $i$, $j$ and define $$H_{t}=\left(\frac{\widehat{\mathcal{X}}_{(:,i_{t},k)}\widehat{\mathcal{Y}}_{(i_{t},:,k)}}{\tau\pi_{i_{t}}}\right)_{(i,j)}=\frac{\widehat{\mathcal{X}}_{(i,i_{t},k)}\widehat{\mathcal{Y}}_{(i_{t},j,k)}}{\tau\pi_{i_{t}}},$$ for $t=1,\cdots,\tau$, and it is easy to verify that
	$$\mathbf{E}[H_{t}]=\sum_{q\in[n]}\pi_{q}\left(\frac{\widehat{\mathcal{X}}_{(i,q,k)}\widehat{\mathcal{Y}}_{(q,j,k)}}{\tau\pi_{q}}\right)=\frac{1}{\tau}(\widehat{\mathcal{X}}_{(k)} \widehat{\mathcal{Y}}_{(k)})_{(i,j)},$$
	and
	$$\mathbf{E}[H_{t}^{2}]=\sum_{q\in[n]}\pi_{q}\left(\frac{\widehat{\mathcal{X}}_{(i,q,k)}\widehat{\mathcal{Y}}_{(q,j,k)}}{\tau\pi_{q}}\right)^{2}=\sum_{q\in[n]}\frac{\widehat{\mathcal{X}}_{(i,q,k)}^{2}\widehat{\mathcal{Y}}_{(q,j,k)}^{2}}{\tau^{2}\pi_{q}}.$$
	Thus,
	\begin{align*}
		\textbf{Var} [H_{t}] =& \mathbf{E}[H_{t}^{2}]-\mathbf{E}[H_{t}]^{2}
		=\sum_{q\in[n]}\frac{\widehat{\mathcal{X}}_{(i,q,k)}^{2}\widehat{\mathcal{Y}}_{(q,j,k)}^{2}}{\tau^{2}\pi_{q}}-\frac{1}{\tau^{2}}(\widehat{\mathcal{X}}_{(k)} \widehat{\mathcal{Y}}_{(k)})_{(i,j)}^{2}.
	\end{align*}
	Since $(\widehat{\mathcal{X}}_{(k)}\widehat{\mathcal{M}}_{(k)}^{H}\widehat{\mathcal{M}}_{(k)}\widehat{\mathcal{Y}}_{(k)})_{(i,j)}=\sum_{t\in[\tau]}H_{t}$, which is the sum of $\tau$ independent random variables, we have that
	$$\mathbf{E}[(\widehat{\mathcal{X}}_{(k)}\widehat{\mathcal{M}}_{(k)}^{H}\widehat{\mathcal{M}}_{(k)}\widehat{\mathcal{Y}}_{(k)})_{(i,j))}]=\sum_{t\in[\tau]}\mathbf{E}[H_{t}]=(\widehat{\mathcal{X}}_{(k)} \widehat{\mathcal{Y}}_{(k)})_{(i,j)},$$
	and
	\begin{align*}
		\textbf{Var} [(\widehat{\mathcal{X}}_{(k)}\widehat{\mathcal{M}}_{(k)}^{H}\widehat{\mathcal{M}}_{(k)}\widehat{\mathcal{Y}}_{(k)})_{(i,j)}] &= \sum_{t\in[\tau]} \textbf{Var} [H_{t}]
		=\sum_{q\in[n]}\frac{\widehat{\mathcal{X}}_{(i,q,k)}^{2}\widehat{\mathcal{Y}}_{(q,j,k)}^{2}}{\tau\pi_{q}}-\frac{1}{\tau}(\widehat{\mathcal{X}}_{(k)} \widehat{\mathcal{Y}}_{(k)})_{(i,j)}^{2}\\
		&\leq\sum_{q\in[n]}\frac{\widehat{\mathcal{X}}_{(i,q,k)}^{2}\widehat{\mathcal{Y}}_{(q,j,k)}^{2}}{\tau\pi_{q}}.
	\end{align*}
	Therefore,
	\begin{align*}
		&\mathbf{E}[\|\mathcal{X}*\mathcal{Y}-\mathcal{X}*\mathcal{M}^{T}*\mathcal{M}*\mathcal{Y}\|_{F}^{2}]
		=\frac{1}{l}\sum_{k=1}^{l}\mathbf{E}[\|\widehat{\mathcal{X}}_{(k)}\widehat{\mathcal{Y}}_{(k)}-\widehat{\mathcal{X}}_{(k)}\widehat{\mathcal{M}}_{(k)}^{H}\widehat{\mathcal{M}}_{(k)}\widehat{\mathcal{Y}}_{(k)}\|_{F}^{2}]\\
		=&\frac{1}{l}\sum_{k=1}^{l}\sum_{i\in[m]}\sum_{j\in[p]}\mathbf{E}[(\widehat{\mathcal{X}}_{(k)}\widehat{\mathcal{Y}}_{(k)}-\widehat{\mathcal{X}}_{(k)}\widehat{\mathcal{M}}_{(k)}^{H}\widehat{\mathcal{M}}_{(k)}\widehat{\mathcal{Y}}_{(k)})_{(i,j)}^{2}]\\
		=&\frac{1}{l}\sum_{k=1}^{l}\sum_{i\in[m]}\sum_{j\in[p]}\textbf{Var} [(\widehat{\mathcal{X}}_{(k)}\widehat{\mathcal{M}}_{(k)}^{H}\widehat{\mathcal{M}}_{(k)}\widehat{\mathcal{Y}}_{(k)})_{(i,j)}]
		\leq \frac{1}{l}\sum_{k=1}^{l}\sum_{i\in[m]}\sum_{j\in[p]}\sum_{q\in[n]}\frac{\widehat{\mathcal{X}}_{(i,q,k)}^{2}\widehat{\mathcal{Y}}_{(q,j,k)}^{2}}{\tau\pi_{q}}\\
		=&\frac{1}{l}\sum_{k=1}^{l}\sum_{q\in[n]}\frac{\|\widehat{\mathcal{X}}_{(:,q,k)}\|_{2}^{2}\|\widehat{\mathcal{Y}}_{(q,:,k)}\|_{2}^{2}}{\tau\pi_{q}}
		\leq\frac{1}{\tau}\sum_{q\in[n]}\frac{l}{\pi_{q}}\left(\frac{1}{l}\sum_{k=1}^{l}\|\widehat{\mathcal{X}}_{(:,q,k)}\|_{2}^{2}\right)\left(\frac{1}{l}\sum_{k=1}^{l}\|\widehat{\mathcal{Y}}_{(q,:,k)}\|_{2}^{2}\right)\\
		=&\frac{1}{\tau}\sum_{q\in[n]}\frac{l}{\pi_{q}}\|\mathcal{X}_{(:,q,:)}\|_{F}^{2}\|\mathcal{Y}_{(q,:,:)}\|_{F}^{2}
		\leq\frac{l}{\beta \tau}\|\mathcal{X}\|_{F}^{2}\|\mathcal{Y}\|_{F}^{2}.
	\end{align*}
\end{proof}

\begin{proof}[Proof of Lemma~{\upshape\ref{fullrank}}] \rm
	Let $\mathcal{M}=\mathcal{D}*\mathcal{S}^{T}$. First, we prove (\ref{fullrank_1}). By Markov's inequality and Lemma \ref{RanMutil}, we obtain that with probability at least $0.9$:
	\begin{align*}
		\|\mathcal{U}_{\mathcal{X}}^{T}*\mathcal{U}_{\mathcal{X}}-\mathcal{U}_{\mathcal{X}}^{T}*\mathcal{M}^{T}*\mathcal{M}*\mathcal{U}_{\mathcal{X}}\|_{F}^2&\leq10\mathbf{E}[\|\mathcal{U}_{\mathcal{X}}^{T}*\mathcal{U}_{\mathcal{X}}-\mathcal{U}_{\mathcal{X}}^{T}*\mathcal{M}^{T}*\mathcal{M}*\mathcal{U}_{\mathcal{X}}\|_{F}^2]\\
		&\leq10\frac{l}{\beta \tau}\|\mathcal{U}_{\mathcal{X}}^{T}\|_{F}^{2}\|\mathcal{U}_{\mathcal{X}}\|_{F}^{2}=10\frac{l}{\beta \tau}p^2.
	\end{align*}
	Hence, for all $i \in [p]$ and $k\in [l]$, we have that

	\begin{align*}
   \left(1-\max_{k\in[l]}\sigma_{i}^{2}(\widehat{\mathcal{M}}_{(k)}\widehat{\mathcal{U}_{\mathcal{X}}}_{(k)})\right)^2&=\left( \lambda_{i}(\widehat{\mathcal{U}_{\mathcal{X}}}_{(k)}^{H}\widehat{\mathcal{U}_{\mathcal{X}}}_{(k)})-\max_{k\in[l]}\lambda_{i}(\widehat{\mathcal{U}_{\mathcal{X}}}_{(k)}^{H}\widehat{\mathcal{M}}_{(k)}^{H}\widehat{\mathcal{M}}_{(k)}\widehat{\mathcal{U}_{\mathcal{X}}}_{(k)})\right)^2\\
	&\leq\max_{k\in[l]}\left( \lambda_{i}(\widehat{\mathcal{U}_{\mathcal{X}}}_{(k)}^{H}\widehat{\mathcal{U}_{\mathcal{X}}}_{(k)})-\lambda_{i}(\widehat{\mathcal{U}_{\mathcal{X}}}_{(k)}^{H}\widehat{\mathcal{M}}_{(k)}^{H}\widehat{\mathcal{M}}_{(k)}\widehat{\mathcal{U}_{\mathcal{X}}}_{(k)})\right)^2\\
	&\leq \max_{k\in[l]}\|\widehat{\mathcal{U}_{\mathcal{X}}}_{(k)}^{H}\widehat{\mathcal{U}_{\mathcal{X}}}_{(k)}-\widehat{\mathcal{U}_{\mathcal{X}}}_{(k)}^{H}\widehat{\mathcal{M}}_{(k)}^{H}\widehat{\mathcal{M}}_{(k)}\widehat{\mathcal{U}_{\mathcal{X}}}_{(k)}\|_{2}^2\\
	&=\|\mathcal{U}_{\mathcal{X}}^{T}*\mathcal{U}_{\mathcal{X}}-\mathcal{U}_{\mathcal{X}}^{T}*\mathcal{M}^{T}*\mathcal{M}*\mathcal{U}_{\mathcal{X}}\|_{2}^2\\
	&\leq l \|\mathcal{U}_{\mathcal{X}}^{T}*\mathcal{U}_{\mathcal{X}}-\mathcal{U}_{\mathcal{X}}^{T}*\mathcal{M}^{T}*\mathcal{M}*\mathcal{U}_{\mathcal{X}}\|_{F}^2\\
	&\leq10\frac{l^2}{\beta \tau}p^{2}\leq\frac{\epsilon}{4}\leq\frac{1}{4},
\end{align*}
	where the penultimate inequality uses the assumed choice of $\tau$.
	Therefore, each singular tubal scalar of $\mathcal{M}*\mathcal{U}_{\mathcal{X}}$ is nonzero, which implies that (\ref{fullrank_1}) holds, that is, $\text{rank}_{t}(\mathcal{D}*\mathcal{S}^{T}*\mathcal{U}_{\mathcal{X}})=\text{rank}_{t}(\mathcal{U}_{\mathcal{X}})=\text{rank}_{t}(\mathcal{X})=p$.
	
	Similar to the proof of Lemma 3.5 in  \cite{tarzanagh2018fast}, we can prove that (\ref{fullrank_2}, (\ref{fullrank_3} and (\ref{fullrank_4} all hold. Here, we omit the details.
\end{proof}

\begin{proof}[Proof of Lemma~{\upshape\ref{multicom}}] \rm
	Let $\mathcal{M}=\mathcal{D}*\mathcal{S}^{T}$. Since $\mathcal{U}_{\mathcal{X}}^{T}*\mathcal{U}_{\mathcal{X}}^{\perp}=O$, where $O$ is zero tubal matrix, we have that
	\begin{align*}
		&\|\mathcal{U}_{\mathcal{X}}^{T}*\mathcal{M}^{T}*\mathcal{M}*\mathcal{U}_{\mathcal{X}}^{\perp}*{\mathcal{U}_{\mathcal{X}}^{\perp}}^{T}*\overrightarrow{\mathcal{Y}}\|_{F}^2=	\|\mathcal{U}_{\mathcal{X}}*\mathcal{U}_{\mathcal{X}}^{T}*\mathcal{M}^{T}*\mathcal{M}*\mathcal{U}_{\mathcal{X}}^{\perp}*{\mathcal{U}_{\mathcal{X}}^{\perp}}^{T}*\overrightarrow{\mathcal{Y}}\|_{F}^2\\
		=&\|\mathcal{U}_{\mathcal{X}}*\mathcal{U}_{\mathcal{X}}^{T}*\mathcal{U}_{\mathcal{X}}^{\perp}*{\mathcal{U}_{\mathcal{X}}^{\perp}}^{T}*\overrightarrow{\mathcal{Y}}-\mathcal{U}_{\mathcal{X}}*\mathcal{U}_{\mathcal{X}}^{T}*\mathcal{M}^{T}*\mathcal{M}*\mathcal{U}_{\mathcal{X}}^{\perp}*{\mathcal{U}_{\mathcal{X}}^{\perp}}^{T}*\overrightarrow{\mathcal{Y}}\|_{F}^2,
	\end{align*}	
which together with Markov's inequality and Lemma \ref{RanMutil} implies that with probability at least $0.9$:
	\begin{align*}
	\|\mathcal{U}_{\mathcal{X}}^{T}*\mathcal{M}^{T}*\mathcal{M}*\mathcal{U}_{\mathcal{X}}^{\perp}*{\mathcal{U}_{\mathcal{X}}^{\perp}}^{T}*\overrightarrow{\mathcal{Y}}\|_{F}^2
	\leq&10\mathbf{E}[\|\mathcal{U}_{\mathcal{X}}*\mathcal{U}_{\mathcal{X}}^{T}*\mathcal{M}^{T}*\mathcal{M}*\mathcal{U}_{\mathcal{X}}^{\perp}*{\mathcal{U}_{\mathcal{X}}^{\perp}}^{T}*\overrightarrow{\mathcal{Y}}\|_{F}^2]\\
	\leq&10\frac{l}{\beta \tau}\|\mathcal{U}_{\mathcal{X}}*\mathcal{U}_{\mathcal{X}}^{T}\|_{F}^2\|\mathcal{U}_{\mathcal{X}}^{\perp}*{\mathcal{U}_{\mathcal{X}}^{\perp}}^{T}*\overrightarrow{\mathcal{Y}}\|_{F}^2\\
	=&10\frac{l}{\beta \tau}p\|\mathcal{U}_{\mathcal{X}}^{\perp}*{\mathcal{U}_{\mathcal{X}}^{\perp}}^{T}*\overrightarrow{\mathcal{Y}}\|_{F}^2\\
	\leq&\frac{\epsilon}{4}\|\mathcal{U}_{\mathcal{X}}^{\perp}*{\mathcal{U}_{\mathcal{X}}^{\perp}}^{T}*\overrightarrow{\mathcal{Y}}\|_{F}^2,
\end{align*}	
	where the last inequality uses the assumed choice of $\tau$.
\end{proof}	
\begin{proof}[Proof of Lemma~{\upshape\ref{multicomp}}] \rm
Let $\mathcal{M}=\mathcal{D}*\mathcal{S}^{T}$, $\mathcal{Q}=\mathcal{U}_{\mathcal{X}}^{\perp}*{\mathcal{U}_{\mathcal{X}}^{\perp}}^{T}*\overrightarrow{\mathcal{Y}}$, and  $i_{1},i_{2},\cdots,i_{\tau}$ be the $\tau$ horizontal slices of $\mathcal{Q}$ that were included in $\mathcal{M}*\mathcal{Q}$. Clearly,
	\begin{align*}
		\mathbf{E}[\|\mathcal{M}*\mathcal{Q}\|_{F}^{2}]&=\mathbf{E}\left[\sum_{t=1}^{\tau}\frac{\|\mathcal{Q}_{(i_{t},:,:)}\|_{F}^{2}}{\tau \pi_{i_{t}}}\right]=\sum_{t=1}^{\tau}\mathbf{E}\left[\frac{\|\mathcal{Q}_{(i_{t},:,:)}\|_{F}^{2}}{\tau \pi_{i_{t}}}\right]=\sum_{t=1}^{\tau}\sum_{i=1}^{n}\pi_{i}\frac{\|\mathcal{Q}_{(i,:,:)}\|_{F}^{2}}{\tau \pi_{i}}=\|\mathcal{Q}\|_{F}^{2},
	\end{align*}	
	which together with  Markov's inequality leads to that
	\begin{align*}
		\|\mathcal{M}*\mathcal{U}_{\mathcal{X}}^{\perp}*{\mathcal{U}_{\mathcal{X}}^{\perp}}^{T}*\overrightarrow{\mathcal{Y}}\|_{F}^2\leq10\mathbf{E}[\|\mathcal{M}*\mathcal{Q}\|_{F}^{2}]=10\|\mathcal{U}_{\mathcal{X}}^{\perp}*{\mathcal{U}_{\mathcal{X}}^{\perp}}^{T}*\overrightarrow{\mathcal{Y}}\|_{F}^{2}
	\end{align*}	
	holds with probability as least $0.9$.
\end{proof}	

\begin{proof}[Proof of Theorem~{\upshape\ref{rel_opt}}] \rm
	Let $\mathcal{M}=\mathcal{D}*\mathcal{S}^{T}$.	If we assume that $\tau\geq40p^{2}l^{2}/\beta\epsilon'$, which makes the assumption on $\tau$ be satisfied for each of Lemmas  \ref{fullrank}, \ref{multicom} and \ref{multicomp}, then the three lemmas hold simultaneously with probability at least $1-3\times(0.1)=0.7$. By (\ref{fullrank_1}) and (\ref{fullrank_3}) of Lemma \ref{fullrank} and the fact that $\mathcal{U}_{\mathcal{X}}*\mathcal{U}_{\mathcal{X}}^{T}+\mathcal{U}_{\mathcal{X}}^{\perp}*{\mathcal{U}_{\mathcal{X}}^{\perp}}^{T}=\mathcal{I}_{n}$, we have that
	\begin{align*}
		\overrightarrow{\mathcal{Y}}-\mathcal{X}*\widetilde{\overrightarrow{\mathcal{B}}}_{\mathcal{W}}&=\overrightarrow{\mathcal{Y}}-\mathcal{X}*(\mathcal{M}*\mathcal{X})^{\dag}*\mathcal{M}*\overrightarrow{\mathcal{Y}}\\
		&=\overrightarrow{\mathcal{Y}}-\mathcal{U}_{\mathcal{X}}*\Sigma_{\mathcal{X}}*\mathcal{V}_{\mathcal{X}}^{T}*\mathcal{V}_{\mathcal{X}}*\Sigma_{\mathcal{X}}^{-1}*(\mathcal{M}*\mathcal{U}_{\mathcal{X}})^{\dag}*\mathcal{M}*\overrightarrow{\mathcal{Y}}\\
		&=\overrightarrow{\mathcal{Y}}-\mathcal{U}_{\mathcal{X}}*(\mathcal{M}*\mathcal{U}_{\mathcal{X}})^{\dag}*\mathcal{M}*\overrightarrow{\mathcal{Y}}\\
		&=\overrightarrow{\mathcal{Y}}-\mathcal{U}_{\mathcal{X}}*(\mathcal{M}*\mathcal{U}_{\mathcal{X}})^{\dag}*\mathcal{M}*\mathcal{U}_{\mathcal{X}}*\mathcal{U}_{\mathcal{X}}^{T}*\overrightarrow{\mathcal{Y}}-\mathcal{U}_{\mathcal{X}}*(\mathcal{M}*\mathcal{U}_{\mathcal{X}})^{\dag}*\mathcal{M}*\mathcal{U}_{\mathcal{X}}^{\perp}*{\mathcal{U}_{\mathcal{X}}^{\perp}}^{T}*\overrightarrow{\mathcal{Y}}\\
		&=\mathcal{U}_{\mathcal{X}}^{\perp}*{\mathcal{U}_{\mathcal{X}}^{\perp}}^{T}*\overrightarrow{\mathcal{Y}}-\mathcal{U}_{\mathcal{X}}*(\mathcal{M}*\mathcal{U}_{\mathcal{X}})^{\dag}*\mathcal{M}*\mathcal{U}_{\mathcal{X}}^{\perp}*{\mathcal{U}_{\mathcal{X}}^{\perp}}^{T}*\overrightarrow{\mathcal{Y}},
	\end{align*}	
	which together with Pythagorean theorem, submultiplicative inequality, and the fact that $\Gamma=(\mathcal{M}*\mathcal{U}_{\mathcal{X}})^{\dag}-(\mathcal{M}*\mathcal{U}_{\mathcal{X}})^{T}$ leads to
	\begin{align*}
		f(\widetilde{\overrightarrow{\mathcal{B}}}_{\mathcal{W}})=&\|\overrightarrow{\mathcal{Y}}-\mathcal{X}*\widetilde{\overrightarrow{\mathcal{B}}}_{\mathcal{W}}\|_{F}^2=\|\mathcal{U}_{\mathcal{X}}^{\perp}*{\mathcal{U}_{\mathcal{X}}^{\perp}}^{T}*\overrightarrow{\mathcal{Y}}-\mathcal{U}_{\mathcal{X}}*(\mathcal{M}*\mathcal{U}_{\mathcal{X}})^{\dag}*\mathcal{M}*\mathcal{U}_{\mathcal{X}}^{\perp}*{\mathcal{U}_{\mathcal{X}}^{\perp}}^{T}*\overrightarrow{\mathcal{Y}}\|_{F}^2\\
		=&\|\mathcal{U}_{\mathcal{X}}^{\perp}*{\mathcal{U}_{\mathcal{X}}^{\perp}}^{T}*\overrightarrow{\mathcal{Y}}-\mathcal{U}_{\mathcal{X}}*(\Gamma+(\mathcal{M}*\mathcal{U}_{\mathcal{X}})^{T})*\mathcal{M}*\mathcal{U}_{\mathcal{X}}^{\perp}*{\mathcal{U}_{\mathcal{X}}^{\perp}}^{T}*\overrightarrow{\mathcal{Y}}\|_{F}^2\\
		=&\|\mathcal{U}_{\mathcal{X}}^{\perp}*{\mathcal{U}_{\mathcal{X}}^{\perp}}^{T}*\overrightarrow{\mathcal{Y}}\|_{F}^2+\|\mathcal{U}_{\mathcal{X}}*(\Gamma+(\mathcal{M}*\mathcal{U}_{\mathcal{X}})^{T})*\mathcal{M}*\mathcal{U}_{\mathcal{X}}^{\perp}*{\mathcal{U}_{\mathcal{X}}^{\perp}}^{T}*\overrightarrow{\mathcal{Y}}\|_{F}^2\\
		\leq&\|\mathcal{U}_{\mathcal{X}}^{\perp}*{\mathcal{U}_{\mathcal{X}}^{\perp}}^{T}*\overrightarrow{\mathcal{Y}}\|_{F}^2+2\|\mathcal{U}_{\mathcal{X}}*(\mathcal{M}*\mathcal{U}_{\mathcal{X}})^{T}*\mathcal{M}*\mathcal{U}_{\mathcal{X}}^{\perp}*{\mathcal{U}_{\mathcal{X}}^{\perp}}^{T}*\overrightarrow{\mathcal{Y}}\|_{F}^2\\
		&+2\|\mathcal{U}_{\mathcal{X}}*\Gamma*\mathcal{M}*\mathcal{U}_{\mathcal{X}}^{\perp}*{\mathcal{U}_{\mathcal{X}}^{\perp}}^{T}*\overrightarrow{\mathcal{Y}}\|_{F}^2\\
		\leq&\|\mathcal{U}_{\mathcal{X}}^{\perp}*{\mathcal{U}_{\mathcal{X}}^{\perp}}^{T}*\overrightarrow{\mathcal{Y}}\|_{F}^2+2\|\mathcal{U}_{\mathcal{X}}^{T}*\mathcal{M}^{T}*\mathcal{M}*\mathcal{U}_{\mathcal{X}}^{\perp}*{\mathcal{U}_{\mathcal{X}}^{\perp}}^{T}*\overrightarrow{\mathcal{Y}}\|_{F}^2\\
		&+2\|\Gamma\|_{2}^2\|\mathcal{M}*\mathcal{U}_{\mathcal{X}}^{\perp}*{\mathcal{U}_{\mathcal{X}}^{\perp}}^{T}*\overrightarrow{\mathcal{Y}}\|_{F}^2,
	\end{align*}	
	which further combined with (\ref{fullrank_4}) of Lemma  \ref{fullrank}, Lemmas \ref{multicom} and \ref{multicomp} gives
	\begin{align*}
		f(\widetilde{\overrightarrow{\mathcal{B}}}_{\mathcal{W}})\leq&(1+\epsilon'/2+10\epsilon')\|\mathcal{U}_{\mathcal{X}}^{\perp}*{\mathcal{U}_{\mathcal{X}}^{\perp}}^{T}*\overrightarrow{\mathcal{Y}}\|_{F}^2
		\leq(1+11\epsilon')f(\widetilde{\overrightarrow{\mathcal{B}}}_{ols}).
	\end{align*}	
	By setting $\epsilon'=\epsilon/11$, we can conclude that (\ref{rele_func}) holds as long as $\{\pi_{i}\}_{i=1}^{n}$ and $\tau$ are as assumed by the theorem.
	
	By (\ref{fullrank_1}) and (\ref{fullrank_3}) of Lemma \ref{fullrank} and the fact that $\mathcal{U}_{\mathcal{X}}*\mathcal{U}_{\mathcal{X}}^{T}+\mathcal{U}_{\mathcal{X}}^{\perp}*{\mathcal{U}_{\mathcal{X}}^{\perp}}^{T}=\mathcal{I}_{n}$, we have that
	\begin{align*}
		\widetilde{\overrightarrow{\mathcal{B}}}_{ols}-\widetilde{\overrightarrow{\mathcal{B}}}_{\mathcal{W}}=&{\mathcal{X}}^{\dag}*\overrightarrow{\mathcal{Y}}-(\mathcal{M}*\mathcal{X})^{\dag}*\mathcal{M}*\overrightarrow{\mathcal{Y}}\\
		=&(\mathcal{U}_{\mathcal{X}}*\Sigma_{\mathcal{X}}*\mathcal{V}_{\mathcal{X}}^{T})^{\dag}*\overrightarrow{\mathcal{Y}}-\mathcal{V}_{\mathcal{X}}*\Sigma_{\mathcal{X}}^{-1}*(\mathcal{M}*\mathcal{U}_{\mathcal{X}})^{\dag}*\mathcal{M}*\overrightarrow{\mathcal{Y}}\\
		=&\mathcal{V}_{\mathcal{X}}*\Sigma_{\mathcal{X}}^{-1}*\mathcal{U}_{\mathcal{X}}^{T}*\overrightarrow{\mathcal{Y}}-\mathcal{V}_{\mathcal{X}}*\Sigma_{\mathcal{X}}^{-1}*(\mathcal{M}*\mathcal{U}_{\mathcal{X}})^{\dag}*\mathcal{M}*\overrightarrow{\mathcal{Y}}\\
		=&\mathcal{V}_{\mathcal{X}}*\Sigma_{\mathcal{X}}^{-1}*\mathcal{U}_{\mathcal{X}}^{T}*\overrightarrow{\mathcal{Y}}-\mathcal{V}_{\mathcal{X}}*\Sigma_{\mathcal{X}}^{-1}*(\mathcal{M}*\mathcal{U}_{\mathcal{X}})^{\dag}*\mathcal{M}*\mathcal{U}_{\mathcal{X}}*{\mathcal{U}_{\mathcal{X}}}^{T}*\overrightarrow{\mathcal{Y}}\\
		&-\mathcal{V}_{\mathcal{X}}*\Sigma_{\mathcal{X}}^{-1}*(\mathcal{M}*\mathcal{U}_{\mathcal{X}})^{\dag}*\mathcal{M}*\mathcal{U}_{\mathcal{X}}^{\perp}*{\mathcal{U}_{\mathcal{X}}^{\perp}}^{T}*\overrightarrow{\mathcal{Y}}\\
		=&-\mathcal{V}_{\mathcal{X}}*\Sigma_{\mathcal{X}}^{-1}*(\mathcal{M}*\mathcal{U}_{\mathcal{X}})^{\dag}*\mathcal{M}*\mathcal{U}_{\mathcal{X}}^{\perp}*{\mathcal{U}_{\mathcal{X}}^{\perp}}^{T}*\overrightarrow{\mathcal{Y}},
	\end{align*}	
	which together with submultiplicative inequality, and the fact that $\Gamma=(\mathcal{M}*\mathcal{U}_{\mathcal{X}})^{\dag}-(\mathcal{M}*\mathcal{U}_{\mathcal{X}})^{T}$ leads to
	\begin{align*}
		\|\widetilde{\overrightarrow{\mathcal{B}}}_{ols}-\widetilde{\overrightarrow{\mathcal{B}}}_{\mathcal{W}}\|_{F}^2
		=&\|\mathcal{V}_{\mathcal{X}}*\Sigma_{\mathcal{X}}^{-1}*(\mathcal{M}*\mathcal{U}_{\mathcal{X}})^{\dag}*\mathcal{M}*\mathcal{U}_{\mathcal{X}}^{\perp}*{\mathcal{U}_{\mathcal{X}}^{\perp}}^{T}*\overrightarrow{\mathcal{Y}}\|_{F}^2\\
		=&\|\Sigma_{\mathcal{X}}^{-1}*((\mathcal{M}*\mathcal{U}_{\mathcal{X}})^{T}+\Gamma)*\mathcal{M}*\mathcal{U}_{\mathcal{X}}^{\perp}*{\mathcal{U}_{\mathcal{X}}^{\perp}}^{T}*\overrightarrow{\mathcal{Y}}\|_{F}^2\\
		\leq&2\|\Sigma_{\mathcal{X}}^{-1}\|_{2}^2\|(\mathcal{M}*\mathcal{U}_{\mathcal{X}})^{T}*\mathcal{M}*\mathcal{U}_{\mathcal{X}}^{\perp}*{\mathcal{U}_{\mathcal{X}}^{\perp}}^{T}*\overrightarrow{\mathcal{Y}}\|_{F}^2+2\|\Sigma_{\mathcal{X}}^{-1}\|_{2}^2\|\Gamma*\mathcal{M}*\mathcal{U}_{\mathcal{X}}^{\perp}*{\mathcal{U}_{\mathcal{X}}^{\perp}}^{T}*\overrightarrow{\mathcal{Y}}\|_{F}^2\\
		\leq&\frac{1}{\sigma_{\min}^2(\text{bcirc}(\mathcal{X}))}2\|\mathcal{U}_{\mathcal{X}}^{T}*\mathcal{M}^{T}*\mathcal{M}*\mathcal{U}_{\mathcal{X}}^{\perp}*{\mathcal{U}_{\mathcal{X}}^{\perp}}^{T}*\overrightarrow{\mathcal{Y}}\|_{F}^2\\
		&+\frac{1}{\sigma_{\min}^2(\text{bcirc}(\mathcal{X}))}2\|\Gamma\|_{2}^2\|\mathcal{M}*\mathcal{U}_{\mathcal{X}}^{\perp}*{\mathcal{U}_{\mathcal{X}}^{\perp}}^{T}*\overrightarrow{\mathcal{Y}}\|_{F}^2,
	\end{align*}	
	which further combined with (\ref{fullrank_4}) of Lemma  \ref{fullrank}, Lemmas \ref{multicom} and \ref{multicomp} gives
	\begin{align*}
		\|\widetilde{\overrightarrow{\mathcal{B}}}_{ols}-\widetilde{\overrightarrow{\mathcal{B}}}_{\mathcal{W}}\|_{F}^2
		\leq&\frac{1}{\sigma_{\min}^2(\text{bcirc}(\mathcal{X}))}(\epsilon'/2+10\epsilon')\|\mathcal{U}_{\mathcal{X}}^{\perp}*{\mathcal{U}_{\mathcal{X}}^{\perp}}^{T}*\overrightarrow{\mathcal{Y}}\|_{F}^2
		\leq\frac{11\epsilon'}{\sigma_{\min}^2(\text{bcirc}(\mathcal{X}))}f(\widetilde{\overrightarrow{\mathcal{B}}}_{ols}).
	\end{align*}	
	 By setting $\epsilon'=\epsilon/11$, we can conclude that (\ref{rele_beta_func}) holds as long as $\{\pi_{i}\}_{i=1}^{n}$ and $\tau$ are as assumed by the theorem.
	
	If we make the additional assumption that $\|\mathcal{U}_{\mathcal{X}}*\mathcal{U}_{\mathcal{X}}^{T}*\overrightarrow{\mathcal{Y}}\|_{F}^2\geq\gamma\|\overrightarrow{\mathcal{Y}}\|_{F}^2$, for some fixed $\gamma\in(0,1]$, then it follows that
	\begin{align*}
		\|\widetilde{\overrightarrow{\mathcal{B}}}_{ols}-\widetilde{\overrightarrow{\mathcal{B}}}_{\mathcal{W}}\|_{F}^2
		=&\frac{\epsilon}{\sigma_{\min}^2(\text{bcirc}(\mathcal{X}))}f(\widetilde{\overrightarrow{\mathcal{B}}}_{ols})
		=\frac{\epsilon}{\sigma_{\min}^2(\text{bcirc}(\mathcal{X}))}\|\mathcal{U}_{\mathcal{X}}^{\perp}*{\mathcal{U}_{\mathcal{X}}^{\perp}}^{T}*\overrightarrow{\mathcal{Y}}\|_{F}^{2}\\
		=&\frac{\epsilon}{\sigma_{\min}^2(\text{bcirc}(\mathcal{X}))}\left(\|\overrightarrow{\mathcal{Y}}\|_{F}^{2}-\|\mathcal{U}_{\mathcal{X}}*{\mathcal{U}_{\mathcal{X}}}^{T}*\overrightarrow{\mathcal{Y}}\|_{F}^{2}\right)\\
		\leq&\frac{\epsilon}{\sigma_{\min}^2(\text{bcirc}(\mathcal{X}))}(\gamma^{-1}-1)\|\mathcal{U}_{\mathcal{X}}*{\mathcal{U}_{\mathcal{X}}}^{T}*\overrightarrow{\mathcal{Y}}\|_{F}^{2}\\
		\leq&\epsilon\frac{\sigma_{\max}^2(\text{bcirc}(\mathcal{X}))}{\sigma_{\min}^2(\text{bcirc}(\mathcal{X}))}(\gamma^{-1}-1)\|\widetilde{\overrightarrow{\mathcal{B}}}_{ols}\|_{F}^{2}
		=\epsilon\kappa^2(\text{bcirc}(\mathcal{X}))(\gamma^{-1}-1)\|\widetilde{\overrightarrow{\mathcal{B}}}_{ols}\|_{F}^{2},
	\end{align*}	
	where the last inequality is due to the fact that
	\begin{align*}
		\|\widetilde{\overrightarrow{\mathcal{B}}}_{ols}\|_{F}^2=&\|{\mathcal{X}}^{\dag}*\overrightarrow{\mathcal{Y}}\|_{F}^2
		=\|(\mathcal{U}_{\mathcal{X}}*\Sigma_{\mathcal{X}}*\mathcal{V}_{\mathcal{X}}^{T})^{\dag}*\overrightarrow{\mathcal{Y}}\|_{F}^2
		=\|\mathcal{V}_{\mathcal{X}}*\Sigma_{\mathcal{X}}^{-1}*\mathcal{U}_{\mathcal{X}}^{T}*\overrightarrow{\mathcal{Y}}\|_{F}^2\\
		=&\|\Sigma_{\mathcal{X}}^{-1}*\mathcal{U}_{\mathcal{X}}^{T}*\overrightarrow{\mathcal{Y}}\|_{F}^2
		\geq\frac{1}{\sigma_{\max}^2(\text{bcirc}(\mathcal{X}))}\|\mathcal{U}_{\mathcal{X}}^{T}*\overrightarrow{\mathcal{Y}}\|_{F}^2
		=\frac{1}{\sigma_{\max}^2(\text{bcirc}(\mathcal{X}))}\|\mathcal{U}_{\mathcal{X}}* \mathcal{U}_{\mathcal{X}}^{T}*\overrightarrow{\mathcal{Y}}\|_{F}^2.
	\end{align*}		
In summary, we have finished the proof of this theorem.
\end{proof}

\begin{proof}[Proof of Lemma~{\upshape\ref{taylor_expan}}] \rm
	By performing a Taylor expansion of  $\widetilde{\overrightarrow{\mathcal{B}}}_{\mathcal{W}}(\overrightarrow{\mathcal{W}})$ around the point $\overrightarrow{\mathcal{W}}_{0}=\overrightarrow{\mathbf{1}}$, we have that
	\begin{align*}
		\widetilde{\overrightarrow{\mathcal{B}}}_{\mathcal{W}}(\overrightarrow{\mathcal{W}})
		=&\widetilde{\overrightarrow{\mathcal{B}}}_{\mathcal{W}}(\overrightarrow{\mathbf{1}})+\frac{\partial\widetilde{\overrightarrow{\mathcal{B}}}_{\mathcal{W}}(\overrightarrow{\mathcal{W}})}{\partial\overrightarrow{\mathcal{W}}^{ST}}\mid_{\overrightarrow{\mathcal{W}}=\overrightarrow{\mathbf{1}}}*(\overrightarrow{\mathcal{W}}-\overrightarrow{\mathbf{1}})+\mathcal{R}_{\mathcal{W}}\\
		=&\widetilde{\overrightarrow{\mathcal{B}}}_{ols}+\frac{\partial(\mathcal{X}^{T}*\text{diag}_{t}(\overrightarrow{\mathcal{W}})*\mathcal{X})^{-1}*\mathcal{X}^{T}*\text{diag}_{t}(\overrightarrow{\mathcal{W}})*\overrightarrow{\mathcal{Y}}}{\partial\overrightarrow{\mathcal{W}}^{ST}}\mid_{\overrightarrow{\mathcal{W}}=\overrightarrow{\mathbf{1}}}*(\overrightarrow{\mathcal{W}}-\overrightarrow{\mathbf{1}})+\mathcal{R}_{\mathcal{W}},
	\end{align*}	
	where the remainder $\mathcal{R}_{\mathcal{W}}$ satisfying $\|\mathcal{R}_{\mathcal{W}}\|_{F}=o_{p}(\|\overrightarrow{\mathcal{W}}-\overrightarrow{\mathbf{1}}\|_{F})$ when $\overrightarrow{\mathcal{W}}$ is close to $\overrightarrow{\mathbf{1}}$.
	By di?erentiation by parts, we obtain that
	\begin{align}
		&\frac{\partial(\mathcal{X}^{T}*\text{diag}_{t}(\overrightarrow{\mathcal{W}})*\mathcal{X})^{-1}*\mathcal{X}^{T}*\text{diag}_{t}(\overrightarrow{\mathcal{W}})*\overrightarrow{\mathcal{Y}}}{\partial\overrightarrow{\mathcal{W}}^{ST}}\nonumber\\
		=&\frac{\partial\text{vec}_{t}((\mathcal{X}^{T}*\text{diag}_{t}(\overrightarrow{\mathcal{W}})*\mathcal{X})^{-1}*\mathcal{X}^{T}*\text{diag}_{t}(\overrightarrow{\mathcal{W}})*\overrightarrow{\mathcal{Y}})}{\partial\overrightarrow{\mathcal{W}}^{ST}}\nonumber \\
		=&(\mathcal{I}_{1}\otimes_{t}(\mathcal{X}^{T}*\text{diag}_{t}(\overrightarrow{\mathcal{W}})*\mathcal{X})^{-1})*\frac{\partial\text{vec}_{t}(\mathcal{X}^{T}*\text{diag}_{t}(\overrightarrow{\mathcal{W}})*\overrightarrow{\mathcal{Y}})}{\partial\overrightarrow{\mathcal{W}}^{ST}}  \label{ppppp1}\\
		&+((\overrightarrow{\mathcal{Y}}^{ST}*\text{diag}_{t}(\overrightarrow{\mathcal{W}})*\mathcal{X}^{R})\otimes_{t}\mathcal{I}_{p})*\frac{\partial\text{vec}_{t}((\mathcal{X}^{T}*\text{diag}_{t}(\overrightarrow{\mathcal{W}})*\mathcal{X})^{-1})}{\partial\overrightarrow{\mathcal{W}}^{ST}}. \label{ppppp2}
	\end{align}	
	First, (\ref{ppppp1}) can be simplified to
	\begin{align*}
		&	(\mathcal{I}_{1}\otimes_{t}(\mathcal{X}^{T}*\text{diag}_{t}(\overrightarrow{\mathcal{W}})*\mathcal{X})^{-1})*(\overrightarrow{\mathcal{Y}}^{ST}\otimes_{t}\mathcal{X}^{T} )*\frac{\partial\text{vec}_{t}(\text{diag}_{t}(\overrightarrow{\mathcal{W}}))}{\partial\overrightarrow{\mathcal{W}}^{ST}} \nonumber\\ =&	(\overrightarrow{\mathcal{Y}}^{ST}\otimes_{t}((\mathcal{X}^{T}*\text{diag}_{t}(\overrightarrow{\mathcal{W}})*\mathcal{X})^{-1}*\mathcal{X}^{T}))*\frac{\partial\text{vec}_{t}(\text{diag}_{t}(\overrightarrow{\mathcal{W}}))}{\partial\overrightarrow{\mathcal{W}}^{ST}} .
	\end{align*}	
	Second, since
	\begin{align*}
		&	\frac{\partial\text{vec}_{t}((\mathcal{X}^{T}*\text{diag}_{t}(\overrightarrow{\mathcal{W}})*\mathcal{X})^{-1})}{\partial\overrightarrow{\mathcal{W}}^{ST}}\\
		=&\frac{\partial\text{vec}_{t}((\mathcal{X}^{T}*\text{diag}_{t}(\overrightarrow{\mathcal{W}})*\mathcal{X})^{-1})}{\partial\text{vec}_{t}(\mathcal{X}^{T}*\text{diag}_{t}(\overrightarrow{\mathcal{W}})*\mathcal{X})}*\frac{\partial\text{vec}_{t}(\mathcal{X}^{T}*\text{diag}_{t}(\overrightarrow{\mathcal{W}})*\mathcal{X})}{\partial\overrightarrow{\mathcal{W}}^{ST}}\\
		=&	(-((\mathcal{X}^{T}*\text{diag}_{t}(\overrightarrow{\mathcal{W}})*\mathcal{X})^{-1})^{ST}\otimes_{t}(\mathcal{X}^{T}*\text{diag}_{t}(\overrightarrow{\mathcal{W}})*\mathcal{X})^{-1})*	\frac{\partial\text{vec}_{t}(\mathcal{X}^{T}*\text{diag}_{t}(\overrightarrow{\mathcal{W}})*\mathcal{X})}{\partial\overrightarrow{\mathcal{W}}^{ST}}\\
		=&(-((\mathcal{X}^{T}*\text{diag}_{t}(\overrightarrow{\mathcal{W}})*\mathcal{X})^{-1})^{ST}\otimes_{t}(\mathcal{X}^{T}*\text{diag}_{t}(\overrightarrow{\mathcal{W}})*\mathcal{X})^{-1})*(\mathcal{X}^{ST}\otimes_{t}\mathcal{X}^{T})*\frac{\partial\text{vec}_{t}(\text{diag}_{t}(\overrightarrow{\mathcal{W}}))}{\partial\overrightarrow{\mathcal{W}}^{ST}}\\
		=&-(((\mathcal{X}^{T}*\text{diag}_{t}(\overrightarrow{\mathcal{W}})*\mathcal{X})^{-1})^{ST}*\mathcal{X}^{ST})\otimes_{t}((\mathcal{X}^{T}*\text{diag}_{t}(\overrightarrow{\mathcal{W}})*\mathcal{X})^{-1}*\mathcal{X}^{T})*	\frac{\partial\text{vec}_{t}(\text{diag}_{t}(\overrightarrow{\mathcal{W}}))}{\partial\overrightarrow{\mathcal{W}}^{ST}},
	\end{align*}
	we have that
	\begin{align*}
		&\frac{\partial(\mathcal{X}^{T}*\text{diag}_{t}(\overrightarrow{\mathcal{W}})*\mathcal{X})^{-1}*\mathcal{X}^{T}*\text{diag}_{t}(\overrightarrow{\mathcal{W}})*\overrightarrow{\mathcal{Y}}}{\partial\overrightarrow{\mathcal{W}}^{ST}}\\
		=&(\overrightarrow{\mathcal{Y}}^{ST}-(\overrightarrow{\mathcal{Y}}^{ST}*\text{diag}_{t}(\overrightarrow{\mathcal{W}})*\mathcal{X}^{R})*((\mathcal{X}^{T}*\text{diag}_{t}(\overrightarrow{\mathcal{W}})*\mathcal{X})^{-1})^{ST}*\mathcal{X}^{ST})\otimes_{t}((\mathcal{X}^{T}*\text{diag}_{t}(\overrightarrow{\mathcal{W}})*\mathcal{X})^{-1}\\
		&*\mathcal{X}^{T})*\frac{\partial\text{vec}_{t}(\text{diag}_{t}(\overrightarrow{\mathcal{W}}))}{\partial\overrightarrow{\mathcal{W}}^{ST}} \\
		=&((\overrightarrow{\mathcal{Y}}^{ST}-(\mathcal{X}*(\mathcal{X}^{T}*\text{diag}_{t}(\overrightarrow{\mathcal{W}})*\mathcal{X})^{-1}*\mathcal{X}^{T}*\text{diag}_{t}(\overrightarrow{\mathcal{W}})*\overrightarrow{\mathcal{Y}})^{ST})\otimes_{t}((\mathcal{X}^{T}*\text{diag}_{t}(\overrightarrow{\mathcal{W}})*\mathcal{X})^{-1}*\mathcal{X}^{T}))\\
		&*\frac{\partial\text{vec}_{t}(\text{diag}_{t}(\overrightarrow{\mathcal{W}}))}{\partial\overrightarrow{\mathcal{W}}^{ST}} \\
		=&((\overrightarrow{\mathcal{Y}}^{ST}-(\mathcal{X}*\widetilde{\overrightarrow{\mathcal{B}}}_{ols})^{ST})\otimes_{t}((\mathcal{X}^{T}*\text{diag}_{t}(\overrightarrow{\mathcal{W}})*\mathcal{X})^{-1}*\mathcal{X}^{T}))*\frac{\partial\text{vec}_{t}(\text{diag}_{t}(\overrightarrow{\mathcal{W}}))}{\partial\overrightarrow{\mathcal{W}}^{ST}}.
	\end{align*}	
	So combining all above results, we can get that
	\begin{align*}
		\widetilde{\overrightarrow{\mathcal{B}}}_{\mathcal{W}}(\overrightarrow{\mathcal{W}})=&\widetilde{\overrightarrow{\mathcal{B}}}_{ols}+\frac{\partial(\mathcal{X}^{T}*\text{diag}_{t}(\overrightarrow{\mathcal{W}})*\mathcal{X})^{-1}*\mathcal{X}^{T}*\text{diag}_{t}(\overrightarrow{\mathcal{W}})*\overrightarrow{\mathcal{Y}}}{\partial\overrightarrow{\mathcal{W}}^{ST}}\mid_{\overrightarrow{\mathcal{W}}=\overrightarrow{\mathbf{1}}}*(\overrightarrow{\mathcal{W}}-\overrightarrow{\mathbf{1}})+\mathcal{R}_{\mathcal{W}}\\
		=&\widetilde{\overrightarrow{\mathcal{B}}}_{ols}+((\overrightarrow{\mathcal{Y}}^{ST}-(\mathcal{X}*\widetilde{\overrightarrow{\mathcal{B}}}_{ols})^{ST})\otimes_{t}((\mathcal{X}^{T}*\text{diag}_{t}(\overrightarrow{\mathcal{W}})*\mathcal{X})^{-1}*\mathcal{X}^{T}))\\
		&*\frac{\partial\text{vec}_{t}(\text{diag}_{t}(\overrightarrow{\mathcal{W}}))}{\partial\overrightarrow{\mathcal{W}}^{ST}} \mid_{\overrightarrow{\mathcal{W}}=\overrightarrow{\mathbf{1}}}*(\overrightarrow{\mathcal{W}}-\overrightarrow{\mathbf{1}})+\mathcal{R}_{\mathcal{W}}\\
		=&\widetilde{\overrightarrow{\mathcal{B}}}_{\text{ols}}+(\widetilde{\overrightarrow{\mathcal{E}}}^{ST}\otimes_{t}(\mathcal{X}^{T}*\mathcal{X})^{-1}*\mathcal{X}^{T})*\begin{bmatrix}
			\mathcal{I}_{n(:,1,:)} * \mathcal{I}_{n(:,1,:)}^{T}\\
			\mathcal{I}_{n(:,2,:)} *	\mathcal{I}_{n(:,2,:)}^{T}   \\
			\vdots   \\
			\mathcal{I}_{n(:,n,:)} *	\mathcal{I}_{n(:,n,:)}^{T} \\
		\end{bmatrix}*(\overrightarrow{\mathcal{W}}-\overrightarrow{\mathbf{1}})+\mathcal{R}_{\mathcal{W}}\\
		=&\widetilde{\overrightarrow{\mathcal{B}}}_{\text{ols}}+(\mathcal{X}^{T}*\mathcal{X})^{-1}*\mathcal{X}^{T}*\text{diag}_{t}(\widetilde{\overrightarrow{\mathcal{E}}})*(\overrightarrow{\mathcal{W}}-\overrightarrow{\mathbf{1}})+\mathcal{R}_{\mathcal{W}},
	\end{align*}	
	where $\widetilde{\overrightarrow{\mathcal{E}}}=\overrightarrow{\mathcal{Y}}-\mathcal{X}*\widetilde{\overrightarrow{\mathcal{B}}}_{ols}$ is the TLS residual tubal vector.
\end{proof}

\begin{proof}[Proof of Theorem~{\upshape\ref{exvar_sta}}] \rm
	Since we perform random sampling with replacement in the Subsamping TLS method, i.e., Algorithm \ref{Subsample_TLS}, it is easy to see that $\overrightarrow{\mathcal{W}}_{(1)} = (w_1 , w_2 , \cdots , w_n )^{T}$ has a scaled multinomial distribution, thus
 $\mathbf{E}[\overrightarrow{\mathcal{W}}_{(1)}]=\mathbf{1}_{n}$ and
 $$\textbf{Var}[\overrightarrow{\mathcal{W}}_{(1)}]=\mathbf{E}[(\overrightarrow{\mathcal{W}}_{(1)}-\mathbf{E}[\overrightarrow{\mathcal{W}}_{(1)}])(\overrightarrow{\mathcal{W}}_{(1)}-\mathbf{E}[\overrightarrow{\mathcal{W}}_{(1)}])^{T}]=\text{diag}\left(\frac{1}{\tau\pi_{i}}\right)-\frac{1}{\tau}J_{n},$$
where $\text{diag}\left(\frac{1}{\tau\pi_{i}}\right)$ is a diagonal matrix with diagonal elements $\left\{\frac{1}{\tau\pi_{i}}\right\}_{i=1}^{n}$  and $J_{n}$ is a $n\times n$ matrix of all ones. Therefore, $\mathbf{E}[\overrightarrow{\mathcal{W}}]=\overrightarrow{\mathbf{1}}$ and
\begin{align*}
	\textbf{Var}[\overrightarrow{\mathcal{W}}]=\mathbf{E}[(\overrightarrow{\mathcal{W}}-\mathbf{E}[\overrightarrow{\mathcal{W}}])*(\overrightarrow{\mathcal{W}}-\mathbf{E}[\overrightarrow{\mathcal{W}}])^{T}]=\text{diag}_{\text{t}}\left(\frac{1}{\tau\pi_{i}}\mathbf{1}\right)-\frac{1}{\tau}\mathcal{J}_{n},
\end{align*}	
where $\mathcal{J}_{n}$ is a tubal matrix whose first frontal slice is a $n\times n$ matrix of all ones and other frontal slices are all zeros.
Now, it is straightforward to calculate conditional expectation and variance of $\widetilde{\overrightarrow{\mathcal{B}}}_{\mathcal{W}}$, that is
	\begin{align*}
		\mathbf{E}_{\overrightarrow{\mathcal{W}}}[\widetilde{\overrightarrow{\mathcal{B}}}_{\mathcal{W}}\mid\overrightarrow{\mathcal{Y}}]&=\mathbf{E}_{\mathcal{W}}[\widetilde{\overrightarrow{\mathcal{B}}}_{\text{ols}}+(\mathcal{X}^{T}*\mathcal{X})^{-1}*\mathcal{X}^{T}*\text{diag}_{t}(\widetilde{\overrightarrow{\mathcal{E}}})*(\overrightarrow{\mathcal{W}}-\overrightarrow{\mathbf{1}})+\mathcal{R}_{\overrightarrow{\mathcal{W}}}\mid\overrightarrow{\mathcal{Y}}]\\
		&=\widetilde{\overrightarrow{\mathcal{B}}}_{\text{ols}}+(\mathcal{X}^{T}*\mathcal{X})^{-1}*\mathcal{X}^{T}*\text{diag}_{t}(\widetilde{\overrightarrow{\mathcal{E}}})*(\mathbf{E}[\overrightarrow{\mathcal{W}}]-\overrightarrow{\mathbf{1}})+\mathbf{E}_{\overrightarrow{\mathcal{W}}}[\mathcal{R}_{\mathcal{W}}]\\
		&=\widetilde{\overrightarrow{\mathcal{B}}}_{\text{ols}}+\mathbf{E}_{\overrightarrow{\mathcal{W}}}[\mathcal{R}_{\mathcal{W}}],
	\end{align*}	
and
	\begin{align*}
		&\mathbf{Var}_{\overrightarrow{\mathcal{W}}}[\widetilde{\overrightarrow{\mathcal{B}}}_{\mathcal{W}}\mid\overrightarrow{\mathcal{Y}}]=\mathbf{Var}_{\overrightarrow{\mathcal{W}}}[\widetilde{\overrightarrow{\mathcal{B}}}_{\text{ols}}+(\mathcal{X}^{T}*\mathcal{X})^{-1}*\mathcal{X}^{T}*\text{diag}_{t}(\widetilde{\overrightarrow{\mathcal{E}}})*(\overrightarrow{\mathcal{W}}-\overrightarrow{\mathbf{1}})+\mathcal{R}_{\mathcal{W}}\mid\overrightarrow{\mathcal{Y}}]\\
		=&\mathbf{Var}_{\overrightarrow{\mathcal{W}}}[(\mathcal{X}^{T}*\mathcal{X})^{-1}*\mathcal{X}^{T}*\text{diag}_{t}(\widetilde{\overrightarrow{\mathcal{E}}})*(\overrightarrow{\mathcal{W}}-\overrightarrow{\mathbf{1}})\mid\overrightarrow{\mathcal{Y}}]+\mathbf{Var}_{\overrightarrow{\mathcal{W}}}[\mathcal{R}_{\mathcal{W}}]\\
		=&(\mathcal{X}^{T}*\mathcal{X})^{-1}*\mathcal{X}^{T}*\text{diag}_{t}(\widetilde{\overrightarrow{\mathcal{E}}})*\mathbf{Var}[\overrightarrow{\mathcal{W}}]*\text{diag}_{t}(\widetilde{\overrightarrow{\mathcal{E}}})*\mathcal{X}*(\mathcal{X}^{T}*\mathcal{X})^{-1}+\mathbf{Var}_{\overrightarrow{\mathcal{W}}}[\mathcal{R}_{\mathcal{W}}]\\
		=&(\mathcal{X}^{T}*\mathcal{X})^{-1}*\mathcal{X}^{T}*\text{diag}_{t}(\widetilde{\overrightarrow{\mathcal{E}}})*\left(\text{diag}_{\text{t}}\left(\frac{1}{\tau\pi_{i}}\mathbf{1}\right)-\frac{1}{\tau}\mathcal{J}_{n}\right)*\text{diag}_{t}(\widetilde{\overrightarrow{\mathcal{E}}})*\mathcal{X}*(\mathcal{X}^{T}*\mathcal{X})^{-1}\\
		&+\mathbf{Var}_{\overrightarrow{\mathcal{W}}}[\mathcal{R}_{\mathcal{W}}]\\
		=&(\mathcal{X}^{T}*\mathcal{X})^{-1}*\mathcal{X}^{T}*\text{diag}_{t}(\widetilde{\overrightarrow{\mathcal{E}}})*\text{diag}_{t}\left(\frac{1}{\tau\pi_{i}}\mathbf{1}\right)*\text{diag}_{t}(\widetilde{\overrightarrow{\mathcal{E}}})*\mathcal{X}*(\mathcal{X}^{T}*\mathcal{X})^{-1}+\mathbf{Var}_{\overrightarrow{\mathcal{W}}}[\mathcal{R}_{\mathcal{W}}].
	\end{align*}	
We next establish the unconditional expectation and variance of $\widetilde{\overrightarrow{\mathcal{B}}}_{\mathcal{W}}$, that is
	\begin{align*}
		\mathbf{E}[\widetilde{\overrightarrow{\mathcal{B}}}_{\mathcal{W}}]&=\mathbf{E}_{\overrightarrow{\mathcal{Y}}}[\mathbf{E}_{\overrightarrow{\mathcal{W}}}[\widetilde{\overrightarrow{\mathcal{B}}}_{\mathcal{W}}\mid\overrightarrow{\mathcal{Y}}]]=\mathbf{E}_{\overrightarrow{\mathcal{Y}}}[\widetilde{\overrightarrow{\mathcal{B}}}_{\text{ols}}+\mathbf{E}_{\overrightarrow{\mathcal{W}}}[\mathcal{R}_{\mathcal{W}}]]
		=\overrightarrow{\mathcal{B}}_{0}+\mathbf{E}[\mathcal{R}_{\mathcal{W}}],
	\end{align*}	
and
	\begin{align*}
		&\mathbf{Var}[\widetilde{\overrightarrow{\mathcal{B}}}_{\mathcal{W}}]=\mathbf{Var}_{\overrightarrow{\mathcal{Y}}}[\mathbf{E}_{\overrightarrow{\mathcal{W}}}[\widetilde{\overrightarrow{\mathcal{B}}}_{\mathcal{W}}\mid\overrightarrow{\mathcal{Y}}]]+\mathbf{E}_{\overrightarrow{\mathcal{Y}}}[\mathbf{Var}_{\overrightarrow{\mathcal{W}}}[\widetilde{\overrightarrow{\mathcal{B}}}_{\mathcal{W}}\mid\overrightarrow{\mathcal{Y}}]]\\
		=&\mathbf{Var}_{\overrightarrow{\mathcal{Y}}}[\widetilde{\overrightarrow{\mathcal{B}}}_{\text{ols}}+\mathbf{E}_{\overrightarrow{\mathcal{W}}}[\mathcal{R}_{\mathcal{W}}]]+\mathbf{E}_{\overrightarrow{\mathcal{Y}}}[(\mathcal{X}^{T}*\mathcal{X})^{-1}*\mathcal{X}^{T}*\text{diag}_{t}(\widetilde{\overrightarrow{\mathcal{E}}})*\text{diag}_{t}\left(\frac{1}{\tau\pi_{i}}\mathbf{1}\right)*\text{diag}_{t}(\widetilde{\overrightarrow{\mathcal{E}}})\\
		&*\mathcal{X}*(\mathcal{X}^{T}*\mathcal{X})^{-1}+\mathbf{Var}_{\overrightarrow{\mathcal{W}}}[\mathcal{R}_{\mathcal{W}}]]\\
		=&\sigma^{2}(\mathcal{X}^{T}*\mathcal{X})^{-1}+(\mathcal{X}^{T}*\mathcal{X})^{-1}*\mathcal{X}^{T}*\mathbf{E}_{\overrightarrow{\mathcal{Y}}}[\text{diag}_{t}(\widetilde{\overrightarrow{\mathcal{E}}})*\text{diag}_{t}\left(\frac{1}{\tau\pi_{i}}\mathbf{1}\right)*\text{diag}_{t}(\widetilde{\overrightarrow{\mathcal{E}}})]*\mathcal{X}\\
		&*(\mathcal{X}^{T}*\mathcal{X})^{-1}+\mathbf{Var}[\mathcal{R}_{\mathcal{W}}]\\
		=&\sigma^{2}(\mathcal{X}^{T}*\mathcal{X})^{-1}+\frac{\sigma^{2}}{\tau}(\mathcal{X}^{T}*\mathcal{X})^{-1}*\mathcal{X}^{T}*\text{diag}_{\text{t}}\left(\frac{1}{\pi_{i}}(\mathbf{1}-	\mathcal{X}_{(i,:,:)}*(\mathcal{X}^{T}*\mathcal{X})^{-1}*\mathcal{X}_{(i,:,:)}^{T})\right)\\
		&*\mathcal{X}*(\mathcal{X}^{T}*\mathcal{X})^{-1}+\mathbf{Var}[\mathcal{R}_{\mathcal{W}}],
	\end{align*}	
where the last equality is from
	\begin{align*}
		\mathbf{E}_{\overrightarrow{\mathcal{Y}}}[\text{diag}_{t}(\widetilde{\overrightarrow{\mathcal{E}}})*\text{diag}_{t}\left(\frac{1}{\tau\pi_{i}}\mathbf{1}\right)*\text{diag}_{t}(\widetilde{\overrightarrow{\mathcal{E}}})]
		=&\text{diag}_{t}\left(\frac{1}{\tau\pi_{i}}\mathbf{1}\right)*\mathbf{E}_{\overrightarrow{\mathcal{Y}}}[\text{diag}_{t}(\widetilde{\overrightarrow{\mathcal{E}}})*\text{diag}_{t}(\widetilde{\overrightarrow{\mathcal{E}}})]\\
		=&\frac{\sigma^{2}}{\tau}\text{diag}_{\text{t}}\left(\frac{1}{\pi_{i}}(\mathbf{1}-	\mathcal{X}_{(i,:,:)}*(\mathcal{X}^{T}*\mathcal{X})^{-1}*\mathcal{X}_{(i,:,:)}^{T})\right).
	\end{align*}	
\end{proof}

\begin{proof}[Proof of Theorem~{\upshape\ref{optimalprob}}] \rm
	Considering that $(\mathcal{X}^{T}*\mathcal{X})^{-1}$ is a T-symmetric T-positive semide?nite tubal matrix and does not depend on $\{\pi_{i}\}_{i=1}^{n}$ , we can convert minimizing $\text{trace}(\mathbf{Var}[\widetilde{\overrightarrow{\mathcal{B}}}_{\mathcal{W}}])$ into minimizing $\text{trace}(\mathcal{N})$,
	where
	$$\mathcal{N}:=\mathcal{X}^{T}*\text{diag}_{\text{t}}\left(\frac{1}{\pi_{i}}(\mathbf{1}-	\mathcal{X}_{(i,:,:)}*(\mathcal{X}^{T}*\mathcal{X})^{-1}*\mathcal{X}_{(i,:,:)}^{T})\right)*\mathcal{X}.$$
	By applying H$\ddot{\text{o}}$lder¡¯s inequality, we can get that
	\begin{align*}
		\text{trace}(\mathcal{N})=&\frac{1}{l}\sum_{k=1}^{l}\text{trace}(\widehat{\mathcal{N}}_{(k)})=\frac{1}{l}\sum_{k=1}^{l}\text{trace}\left(\widehat{\mathcal{ X}}_{(k)}^{H}\text{diag}\left(\frac{1}{\pi_{i}}\left(1-\widehat{\mathcal{ X}}_{(i,:,k)}(\widehat{\mathcal{ X}}_{(k)}^{H}\widehat{\mathcal{ X}}_{(k)})^{-1}\widehat{\mathcal{ X}}_{(i,:,k)}^{H}\right)\right)\widehat{\mathcal{ X}}_{(k)}\right)\\
		=&\frac{1}{l}\sum_{k=1}^{l}\sum_{i=1}^{n}\left(\frac{1}{\pi_{i}}\left(1-\widehat{\mathcal{ X}}_{(i,:,k)}(\widehat{\mathcal{ X}}_{(k)}^{H}\widehat{\mathcal{ X}}_{(k)})^{-1}\widehat{\mathcal{ X}}_{(i,:,k)}^{T}\right)\|\widehat{\mathcal{ X}}_{(i,:,k)}\|_{2}^2\right)\\
		=&\sum_{i=1}^{n}\frac{1}{\pi_{i}}\left(\frac{1}{l}\sum_{k=1}^{l}\left(1-\|\widehat{\mathcal{U}_{\mathcal{X}}}_{(i,:,k)}\|_{2}^{2}\right)\|\widehat{\mathcal{ X}}_{(i,:,k)}\|_{2}^2\right)\sum_{i=1}^{n}\pi_{i}\\
		\geq&\left(\sum_{i=1}^{n}\sqrt{\frac{1}{l}\sum_{k=1}^{l}\left(1-\|\widehat{\mathcal{U}_{\mathcal{X}}}_{(i,:,k)}\|_{2}^{2}\right)\|\widehat{\mathcal{ X}}_{(i,:,k)}\|_{2}^2}\right)^{2}
	\end{align*}
	with the equality holds if and only if  $\pi_{i}\propto\sqrt{\frac{1}{l}\sum_{k=1}^{l}\left(1-\|\widehat{\mathcal{U}_{\mathcal{X}}}_{(i,:,k)}\|_{2}^{2}\right)\|\widehat{\mathcal{ X}}_{(i,:,k)}\|_{2}^2}.$
\end{proof}
\section{The Fourier version of the Subsampling TLS method\label{app2}}

\begin{algorithm}[htbp]
	\caption{Subsampling TLS in Fourier domain} \label{Subsample_TLS_Four}
	\hspace*{0.02in} {\bf Input:}
	$\mathcal{X}\in \mathbb{K}^{n\times p }_{l}$, $\overrightarrow{\mathcal{Y}}\in \mathbb{K}^{n }_{l}$, sample size $\tau\geq p$  and a probability distribution $\{\pi_{i}\}_{i=1}^{n}$
	\begin{algorithmic}[1]
		\State $\widehat{\mathcal{ X}}=\texttt{fft}(\mathcal{ X},[~],3)$; $\widehat{\overrightarrow{\mathcal{Y}}}=\texttt{fft}(\overrightarrow{\mathcal{Y}},[~],3)$
		\State Initialize $\widehat{\mathcal{S}}\in\mathbb{K}_{l}^{n\times \tau}$ and $\widehat{\mathcal{D}}\in\mathbb{K}_{l}^{\tau \times \tau}$ to zero tubal matrices
		\For{$t=1,\cdots,\tau$}
		\State Pick $i_{t}\in[n]$, where $\mathbf{P}(i_{t}=i)=\pi_{i}$
		\State $\widehat{\mathcal{D}}_{(t,t,1)}=\frac{1}{\sqrt{\tau\pi_{i_{t}}}}$
		\State $\widehat{\mathcal{S}}_{(i_{t},t,1)}=1$
		\EndFor
		\For{$k=1,\cdots,\lceil\frac{l+1}{2}\rceil$}
		\State $\widehat{\widetilde{\overrightarrow{\mathcal{B}}}}_{ \mathcal{W}(k)}=(\widehat{\mathcal{D}}_{(1)}*\widehat{\mathcal{S}}_{(1)}^{T}*\widehat{\mathcal{X}}_{(k)})^{\dag}*\widehat{\mathcal{D}}_{(1)}*\widehat{\mathcal{S}}_{(1)}^{T}*\widehat{\overrightarrow{\mathcal{Y}}}_{(k)}$
		\EndFor
		\For{$k=\lceil\frac{l+1}{2}\rceil+1,\cdots,l$}
	   	\State	$\widehat{\widetilde{\overrightarrow{\mathcal{B}}}}_{ \mathcal{W}(k)}=\text{conj}(\widehat{\widetilde{\overrightarrow{\mathcal{B}}}}_{ \mathcal{W}(l-k+2)})$		
		\EndFor
		\State $\widetilde{\overrightarrow{\mathcal{B}}}_{\mathcal{W}}= \texttt{ifft}\left(\widehat{\widetilde{\overrightarrow{\mathcal{B}}}}_{ \mathcal{W}},[~],3\right)$
	\end{algorithmic}
	\hspace*{0.02in} {\bf Output:} the approximation solution $\widetilde{\overrightarrow{\mathcal{B}}}_{\mathcal{W}}$
\end{algorithm}
\end{document}